\documentclass[mathpazo]{cicp}
\usepackage{amstext}
\usepackage{amsmath}
\usepackage{amsfonts}
\usepackage{algorithm}
\usepackage{amssymb}
\usepackage{algpseudocode}
\usepackage{xcolor}
\usepackage{mathtools}
\def\W11{{{\mathop{W^{1,1}_{2}}\limits^{\bullet\,\,\,\,\,\,\,}}}}

\begin{document}
\title{A finite volume method preserving the invariant region property for the   quasimonotone reaction-diffusion  systems}

\author[O.~Author]{Huifang Zhou\affil{1}\comma\corrauth,  Yuchen, Sun\affil{1}, Fuchang Huo\affil{1},}
\address{\affilnum{1} \ School of Mathematics, Jilin University, Changchun 130012, P.R.  China}

 \emails{{\tt zhouhuifang@jlu.edu.cn} (H.~Zhou), {\tt sunyc21@mails.jlu.edu.cn} (Y.~Sun),
          {\tt huofc22@mails.jlu.edu.cn} (F.~Huo)}

\begin{abstract}
We present	a finite volume method preserving the   invariant region property (IRP) for the   reaction-diffusion systems with quasimonotone  functions, including nondecreasing, decreasing, and mixed quasimonotone systems.  
The diffusion terms and  time derivatives are discretized by  a finite volume method  satisfying the discrete maximum principle (DMP)  and the backward Euler method, respectively. 
The discretization leads to an implicit and nonlinear scheme, and it is proved to  preserve the   invariant region property  unconditionally. We construct an iterative algorithm and prove the invariant region property ar each iteration step.   Numerical examples are shown to confirm the accuracy  and invariant region property of our scheme.
\end{abstract}


\maketitle

\section{Introduction}
\label{sec1}
Reaction-diffusion systems are mathematical models that describe  the behaviors of a large range of physical,  biological, chemical and electrical  phenomena \cite{Justin2012,Hornung1991,Spigler1992,Colli-Franzone2004,Bourgault2009}. 
They are utilized to mimic the variations in chemical substance concentrations caused by local reactions and diffusions in the field of chemistry,  as well as the spread of infectious diseases and population growth  \cite{Hunacek2022}  in   biology;
the neutron diffusion theory and   the Ginzburg–Landau equations for modeling superconductivity \cite{Du1992}  in physics;  and   the FitzHugh-Nagumo model for simulating the  transmission of electrical impulses  in neurology  and so on.

It is of great importance for the numerical methods to preserve the IRP.  The IRP refers to the property of reaction-diffusion systems  that the solution lies in the range of the initial and boundary values, it reflects the physical constraints of the unknown variables. Hence the numerical solution is expected to preserve the IRP as well. Besides, the IRP of numerical schemes has great importance in establishing the  prior  estimates, existence  and stability  of the solution \cite{Smoller1999}. Proposing the IRP-preserving schemes for reaction-diffusion equations is necessary in both physical and mathematical aspects.  
The finite difference methods   \cite{Jerome1984,Conte2023,Gong2023,Li2021} have been widely applied  to solve the reaction-diffusion equations because of simplicity.  \cite{Jerome1984} employs a fully implicit time-discretization method, and  establishes the IRP and stability  of the scheme by  M-matrix analysis. 
\cite{Li2021} combines    exponential time differencing method and overlapping domain decomposition technique   to get a maximum bound principle (MBP) preserving method for the one-component reaction-diffusion equations, where MBP can  be viewed  as    a  special type of invariant region.
In \cite{Conte2023}, the  nonstandard finite difference method combined  with time-accurate and high stable  explicit method is constructed  to obtain the positivity-preserving  scheme for the reaction-diffusion model describing the vegetation evolution in arid environments.
In \cite{Gong2023}, the authors use the  $\theta$-weighted time-stepping scheme and  corresponding iterative  approach for a class of semilinear parabolic equations,  
where  the discrete  MBP  holds  under the  constraint of time step and mesh size.  
However, most finite difference methods are restricted to  rectangular meshes.  
In addition, the finite element method with  implicit–explicit Euler time-discretization is employed to solve the 3D reaction–diffusion systems in \cite{Frittelli2018},  where the IRP is preserved on Delaunay triangular meshes.
The nonlinear Galerkin method is used in \cite{Sembera2001}  to solve the system of reaction–diffusion equations, in which the algorithm needs to calculate the orthonormal basis of the space composed of eigenvectors of the diffusion  operator.  In   \cite{Zhou2022},  the finite volume method preserving the  IRP  is applied to a specific  type of reaction-diffusion systems called FitzHugh-Nagumo equation  on polygonal meshes.    \cite{Du2021}  develops the unified framework  which covers many    numerical schemes   to obtain a MBP-preserving  method for the semilinear parabolic equations.

The goal of this paper is to propose an IRP-preserving finite volume method for the coupled  quasimonotone  parabolic systems on distorted meshes. Compared to our  previous work  that can only handle specific nonlinear reaction terms in   \cite{Zhou2022}, i.e., $f_1(u,v)=u(1-u)(u-a)-v$, $f_2(u,v) = \rho u- \gamma v$,  this work  can handle more general nonlinear reaction terms.  
We utilize the DMP-preserving finite volume scheme  to  discretize the spacial derivatives, and employ the  fully implicit scheme to  discretize the temporal derivatives.  For the problem with three basic types of quasimonotone functions,  we analyze that  the  implicit scheme is unconditionally IRP-preserving and  exists  at least one solution. To solve the nonlinear scheme, we add a specific  linear term  in the iterative algorithm    and  prove the IRP of  the iterative method.  Numerical results demonstrate that  our scheme reaches second-order accuracy in spacial direction and preserves the IRP for different problems.  We also present  the comparison result  with the nine-point scheme  to demonstrate that it cannot preserve the IRP.

This paper is organized as follows. Section 2 gives the model problem and its  corresponding  invariant region theory.   In Section 3, the fully implicit finite volume scheme  is presented and the IRP is discussed. In  Section 4, the iterative approach is introduced and  its  IRP is analyzed.  In Section 5, we give numerical experiments   to demonstrate  the accuracy and preservation of  IRP.  Finally, the Section 6 provides a summary of this paper.

\section{Invariant region theory of the model problem}
\label{sec2}
In this paper, we will investigate the  coupled system of two  parabolic  equations on  a bounded space-time domain $Q_T=\Omega  \times (0,T]$ as  
\begin{align}
	\partial_t u -\nabla \cdot( \kappa_1 \nabla u )&=  f_1(u,v),  & & \text{in}~~ Q_T,  \label{problem_1}
	\\
	\partial_t v -\nabla \cdot( \kappa_2 \nabla v )&=  f_2(u,v),  & & \text{in}~~ Q_T,   \label{problem_2}
\end{align}
subject  to the initial  conditions $u(\mathbf{x},0)  =u_{0}(\mathbf{x})$ and $v(\mathbf{x},0)  =v_{0}(\mathbf{x})$ on $\Omega$ and Dirichlet boundary  conditions  $u(\mathbf{x},t)  =g_1(\mathbf{x},t)$ and  $ v(\mathbf{x},t) =g_2(\mathbf{x},t)$ on $S_T=\partial \Omega \times (0,T]$.     We suppose   $\Omega$ is an open bounded polygonal domain in $\mathbb{R}^2$, $\partial\Omega\in C^{2}$,   $\kappa_1$ and $\kappa_2$  are coercive tensor-valued functions,  $f_1$ and  $f_2$
are  nonlinear  functions of $u$ and $v$.

The notations of standard Sobolev spaces are employed, with $(\cdot,\cdot)$ representing the $L^2(Q_T)$ inner-product.  Define bilinear forms $B_1(u,\phi_1) = (-u\partial_t\phi_1+\kappa_1 \nabla u,\nabla \phi_1)$ and  $B_2(v,\phi_2) = (-v\partial_t\phi_2+\kappa_2 \nabla v,\nabla \phi_2)$.   We say a   function  $(u,v)\in [W_2^{1,1}(Q_T)$
$\bigcap L^{\infty}(Q_T)]^2$ is a weak solution of the problem \eqref{problem_1}-\eqref{problem_2}
provided that

$\mathrm{(i)}$  for any $(\phi_1 ,\phi_2 )\in [\W11(Q_T)]^2$ and $(\phi_1(\mathbf{x},T) ,\phi_2(\mathbf{x},T))=0$ a.e. in $\Omega$, there hold that 
\begin{align*}
	&B_1(u,\phi_1) = ( f_1(u,v),\phi_1),
	\\
	&B_2(v,\phi_2)= ( f_2(u,v),\phi_2);
\end{align*}

$\mathrm{(ii)}$  $u(\mathbf{x},0) = u_0(\mathbf{x})$, $v(\mathbf{x},0) = u_0(\mathbf{x})$ a.e. in $\Omega$ in the sense of trace;

$\mathrm{(iii)}$ $u(\mathbf{x},t) = g_1(\mathbf{x},t)$ ,  $v(\mathbf{x},t) = g_2(\mathbf{x},t)$ a.e. on $S_T$ in the sense of trace.

\begin{definition}
	$(\mathbf{Invariant~ region~ property})$
	A closed subset $\Sigma=[m_1,M_1]\times [m_2,M_2]$ is called an invariant region
	of the problem \eqref{problem_1}-\eqref{problem_2} if for almost  every  $(u_0,v_0)$ and $(g_1,g_2) \in \Sigma$, the corresponding
	solution $(u,v) \in \Sigma$ for all $0<t \leq T$.
\end{definition}

The function $f_1 (u, v)$ is defined to be quasimonotone nondecreasing if  $f_1$ is nondecreasing for fixed $v$.  Similarly, $f_2(u, v)$ is defined to be quasimonotone nondecreasing (resp., nonincreasing) if  $f_2$ is nondecreasing  for fixed $u$. Quasimontone nonincreasing is defined similarly.

\begin{definition}
	$(\mathbf{Quasimonotone~function})$ 
	For a vector of function $\mathbf{f}=\left(f_1, f_2\right)$ of two components, there are three basic types of quasimonotone functions:  if both $f_1$ and $f_2$ are quasimonotone
	nondecreasing (resp., nonincreasing) for $\left(u, v\right) \in \Sigma$, then  $\mathbf{f}$ is  quasimonotone nondecreasing (resp., nonincreasing);   
	if one of $f_1$ and $f_2$ is quasimonotone nonincreasing, and the other one is nondecreasing in $\Sigma$, then  $\mathbf{f}$ is  mixed quasimonotone.
\end{definition}

\begin{definition}
	$(\mathbf{Lipschitz~continuous})$   
	A vector function $\mathbf{f}=\left(f_1, f_2\right)$  is Lipschitz continuous in $\Sigma$ when the following condition is satisfied:
	There exists a constant $\lambda>0$  such that for any $(u_1, v_1),(u_2, v_2) \in \Sigma$, it holds that
	\begin{align*}
		\begin{gathered}
			\left|f_1\left(u_1, v_1\right)-f_1\left(u_2, v_2\right)\right| \leq \lambda \left(\left|u_1-u_2\right|+\left|v_1-v_2\right|\right), \\
			\left|f_2\left(u_1, v_1\right)-f_2\left(u_2, v_2\right)\right| \leq \lambda \left(\left|u_1-u_2\right|+\left|v_1-v_2\right|\right) .
		\end{gathered}
	\end{align*}
\end{definition}

The following lemma demonstrates that the  quasimonotone reaction-diffusion  systems  possess the  invariant region property under certain hypotheses, whose proof can be found in \cite{ Fife1979}.
\begin{lemma}\label{F-N-invariant_region}
	Suppose $\kappa_1, \kappa_2 \in [L^{\infty}(Q_T)]^{2\times2}$  are  uniformly positive definite in $Q_T$, $(u_0, v_0) \in [H^1(\Omega)]^2$ and there exists $(G_1, G_2) \in\left[W^{2,1}\left(Q_T\right)\right]^2$ such that $(G_1, G_2)|_{S_T}=(g_1, g_2)$. Denote $\Sigma=$ $\left[m_1, M_1\right] \times\left[m_2, M_2\right]$, where $m_1, M_1, m_2, M_2$ are constants.
	Suppose  $(u_0,v_0) $ and $(g_1,g_2) \in \Sigma$, $\mathbf{f}=\left(f_1, f_2\right)$ is quasimonotone and Lipschitz continuous in $\Sigma$ and  satisfies the following  relations
	\begin{align}\label{condition_invariant}
		&f_ 1(m_1,v)\geq0, \quad f_ 1(M_1,v)\leq0, 
		\quad f_2(u,m_2)\geq0,  \quad f_ 2(u,M_2)\leq0,   \quad \forall  (u,v) \in \Sigma, 
	\end{align}  
	then  the  coupled system \eqref{problem_1}-\eqref{problem_2}  has a weak solution $(u,v) \in \Sigma$ in $ [W_2^{1,1}(Q_T) \bigcap  L^{\infty}(Q_T)]^2$   and is unique  in $\Sigma$.
\end{lemma}

\begin{remark}
	(1)	 When $\mathbf{f}$ is quasimonotone  nondecreasing, the condition  \eqref{condition_invariant} is  equivalent to $f_ 1(m_1,m_2)\geq0, f_ 1(M_1,M_2)\leq0, 
	f_2(m_1,m_2)\geq0,   f_ 2(M_1,M_2)\leq0$; 
	
	(2) when $\mathbf{f}$ is quasimonotone  nonincreasing, 
	the condition  \eqref{condition_invariant} is  equivalent to $f_ 1(m_1,M_2)$ $ \geq0, f_ 1(M_1,m_2)\leq0, 
	f_2(M_1,m_2)\geq0,   f_ 2(m_1,M_2)\leq0$; 
	
	(3) when $\mathbf{f}$ is  mixed quasimonotone with nonincreasing $f_1$ and nondecreasing $f_2$, the   condition  \eqref{condition_invariant} is  equivalent to $f_ 1(m_1,M_2)\geq0, f_ 1(M_1,m_2)\leq0, 
	f_2(m_1,m_2)\geq0,   f_ 2(M_1,M_2)\leq0$;  
	
	(4) when $\mathbf{f}$ is  mixed quasimonotone with nondecreasing $f_1$ and nonincreasing $f_2$, the   condition  \eqref{condition_invariant} is  equivalent to $f_ 1(m_1,m_2)\geq0, f_ 1(M_1,M_2)\leq0, 
	f_2(M_1,m_2)\geq0,   f_ 2(m_1,M_2)\leq0$.
\end{remark}

\section{The  IRP-preserving finite volume scheme}
In this section, an IRP-preserving finite volume scheme is constructed to solve the coupled semilinear parabolic equations \eqref{problem_1}-\eqref{problem_2}.  The numerical method  uses DMP-preserving  finite volume method in space and backward Euler method in time, and the nonlinear IRP-preserving scheme is obtained.

In order to present the numerical scheme,  it is necessary to introduce  the following   notations, as indicated in Table. \ref{notation} and Fig. \ref{stencil}.  We suppose that  each polygonal cell is star-shaped with respect to its cell-center.
\begin{table}[htbp] 
	\centering  \caption{The notations.}   
	\begin{tabular}{l|l} 
		\hline   
		$K$ or $L$ & the cell  or  the cell-center 
		\\    
		$A$ or $B$ & the vertex of the cell-edge
		\\    $m(K)$ & the area of cell $K$ 
		\\    $h$ &  the maximum diameter of all cells
		\\    $\sigma$ & the  cell-edge 
		\\    $|\sigma|$ & the length of $\sigma$ 
		\\  
		$I$   & the midpoint of $\sigma$
		\\    $\mathcal{J}_{in}$ &  the set of cells \\    $\mathcal{J}_{out}$ & the set of cell-edges on $\partial \Omega$ \\    $\mathcal{J}$ & $\mathcal{J}=\mathcal{J}_{in}\cup \mathcal{J}_{out}$ \\    $\mathcal{E}_{K}$ & the set of cell-edges of $K$ \\    $\mathcal{E}$ & the set  of all  cell-edges 
		\\    
		$\mathbf{n}_{K,\sigma}$ & the unit outward normal vector on $\sigma$ of cell $K$ \\    $\boldsymbol{\tau}_{K I}$ (resp.  $\boldsymbol{\tau}_{L I}$) &  the unit tangential vector of $KI$ (resp. $LI$)
		\\  
		$\boldsymbol{\nu}_{K I}$ (resp. $\boldsymbol{\nu}_{L I}$) & the unit normal vector  of $KI$   (resp. $LI$) 
		\\    
		$\theta_{K,\sigma}$ (resp.  $\theta_{L,\sigma}$) &  the angle between  vectors $\boldsymbol{\tau}_{K I}$ and $\mathbf{n}_{K,\sigma}$ (resp. $\boldsymbol{\tau}_{L I}$ and $\mathbf{n}_{L,\sigma}$) 
		\\
		$t^{n+1}$ & $t^{n+1}= (n+1)\Delta t$, $\Delta t = \frac{T}{N}$
		\\
		$\mathcal{F}^{n+1}_{K, \sigma}$ (resp. $\mathcal{F}^{n+1}_{L, \sigma}$) & the continuous normal flux of $u$ on  edge $\sigma$ of the cell $K$  at $t^{n+1}$(resp. $L$) 
		\\
		$\tilde{\mathcal{F}}^{n+1}_{K, \sigma}$ (resp. $\tilde{\mathcal{F}}^{n+1}_{L, \sigma}$) & the continuous normal flux of $v$ on  edge $\sigma$ of the cell $K$  at $t^{n+1}$(resp. $L$) 
		\\$F^{n+1}_{K, \sigma}$ (resp. $F^{n+1}_{L, \sigma}$) & the discrete normal flux of $u$ on  edge $\sigma$ of the cell $K$ at $t^{n+1}$(resp. $L$) 
		\\$\tilde{F}^{n+1}_{K, \sigma}$ (resp. $\tilde{F}^{n+1}_{L, \sigma}$) & the discrete normal flux of $v$ on  edge $\sigma$ of the cell $K$ at $t^{n+1}$(resp. $L$) 
		\\
		$U_X^{n+1}(X=K,L,A,B,I,\cdots)$ & the discrete solution $U$ defined at the point $X$ at  $t^{n+1}$
		\\  
		$V_X^{n+1}(X=K,L,A,B,I,\cdots)$ & the discrete solution $V$ defined at the point $X$ at  $t^{n+1}$
		\\  
		\hline 
	\end{tabular}  \label{notation}
\end{table}%

\begin{figure}[!htbp]
	\centering
	\includegraphics[width=4in]{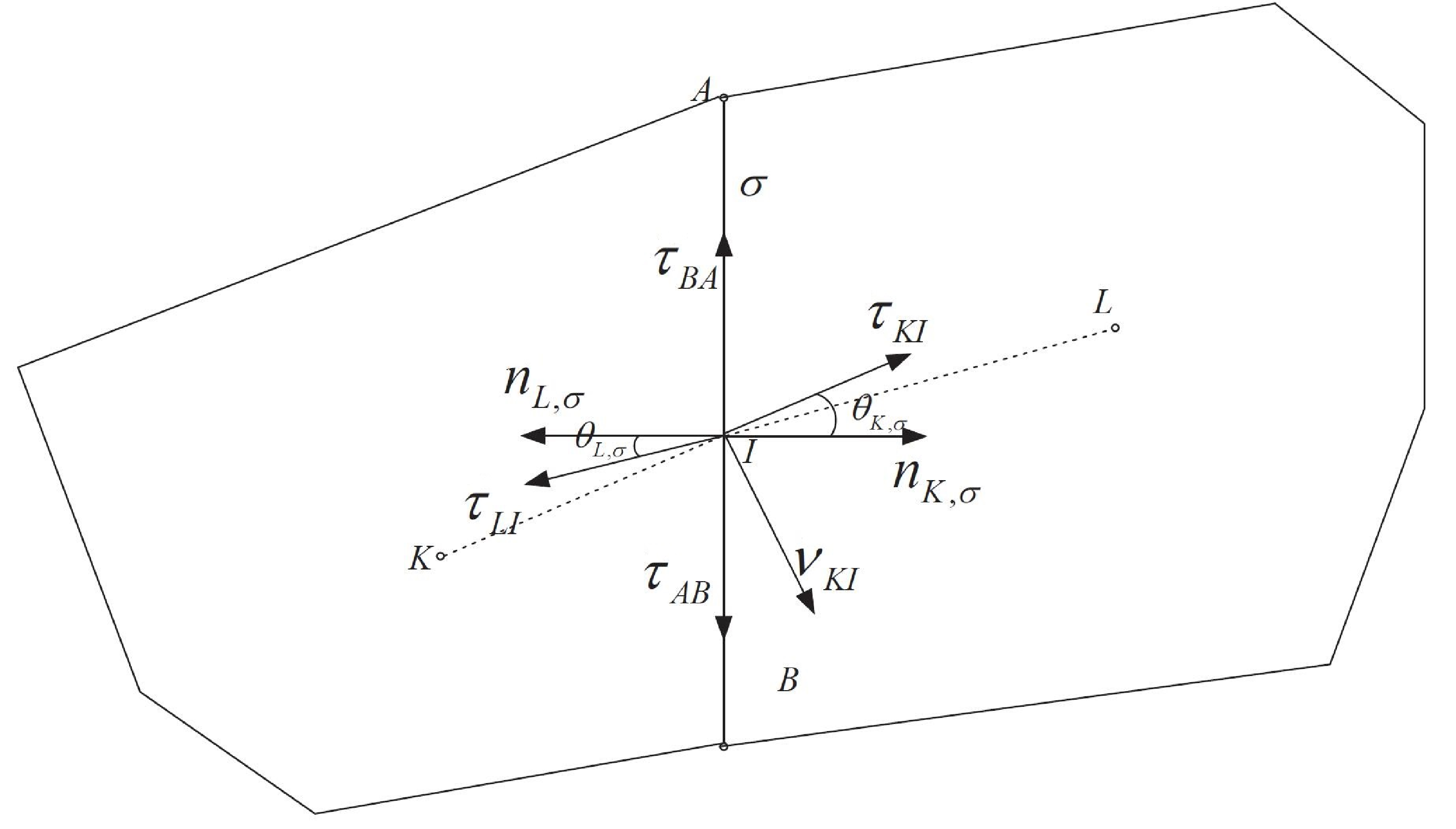}\label{stencil}
	\caption{The notations} 
\end{figure}

Integrating the diffusion parts of \eqref{problem_1}-\eqref{problem_2} over cell $K$ at $t^{n+1}$ and using Green's formula gives 
\begin{align*}
	-\int_K \nabla \cdot( \kappa_1 \nabla u )\big|_{t=t^{n+1}}\mathrm{d}\mathbf{x} =\sum_{\sigma \in \mathcal{E}_{K}} \mathcal{F}^{n+1}_{K, \sigma},
	\\
	-\int_K \nabla \cdot( \kappa_2\nabla v )\big|_{t=t^{n+1}}\mathrm{d}\mathbf{x} =\sum_{\sigma \in \mathcal{E}_{K}} \tilde{\mathcal{F}}^{n+1}_{K, \sigma},
\end{align*}
where $\mathcal{F}^{n+1}_{K, \sigma}=-\int_{\sigma} (\mathbf{n}_{K, \sigma}  \cdot \kappa_1^{\mathrm{T}}\nabla u)\big|_{t=t^{n+1}} \mathrm{d} l$ and $\tilde{\mathcal{F}}^{n+1}_{K, \sigma}=-\int_{\sigma} (\mathbf{n}_{K, \sigma}  \cdot \kappa_2^{\mathrm{T}}\nabla v)\big|_{t=t^{n+1}} \mathrm{d} l$   represent  the continuous normal flux on  edge $\sigma$ for $u$ and $v$, respectively. 

Next, we give the discretization of ${\mathcal{F}}^{n+1}_{K, \sigma}$.  The  discretization of  $\tilde{\mathcal{F}}^{n+1}_{K, \sigma}$  is similar, which is omitted. 
Without ambiguity, the superscript  $n+1$ is also omitted in the rest of this section. 
We employ the DMP-preserving numerical fluxes  proposed in \cite{Sheng2024},
where the numerical fluxes are both nonlinear and conservative. 
The  vector $\kappa_1^{\mathrm{T}} \mathbf{n}_{K, \sigma}$ can be  decomposed into $\kappa_1^{\mathrm{T}} \mathbf{n}_{K, \sigma}=-\alpha_K   {\tau}_{B A}+\beta_K   {\tau}_{K I}$,
where $\theta_{K,\sigma} \in (-\frac{\pi}{2},\frac{\pi}{2})$ denotes   the angle between  vectors $\boldsymbol{\tau}_{K I}$ and $\mathbf{n}_{K,\sigma}$, $ \alpha_K=\frac{1}{\cos \theta_{K, \sigma}} \boldsymbol{\nu}_{K I} \cdot\left(\kappa_1^{\mathrm{T}} \mathbf{n}_{K, \sigma}\right)$, $\beta_K=\frac{1}{\cos \theta_{K, \sigma}} \mathbf{n}_{K, \sigma} \cdot\left(\kappa_1^{\mathrm{T}} \mathbf{n}_{K, \sigma}\right)$,
$\alpha_L=\frac{1}{\cos \theta_{L, \sigma}} \boldsymbol{\nu}_{L I}\cdot\left(\kappa_1^{\mathrm{T}} \mathbf{n}_{L, \sigma}\right)$, $\beta_L=\frac{1}{\cos \theta_{L, \sigma}} \mathbf{n}_{L, \sigma}\cdot\left(\kappa_1^{\mathrm{T}} \mathbf{n}_{L, \sigma}\right)$.
Using Taylor expansion, we obtain
\begin{align}
	\begin{split}
		&	\mathcal{F}_{K, \sigma}=\alpha_{K, \sigma}(u(A)-u(B))-\frac{|A B|}{|I K|} \beta_{K, \sigma}(u(I)-u(K))+O(h^2),  \label{nine-point-flux-K-1}
		\\
		&	\mathcal{F}_{L, \sigma}=\alpha_{L, \sigma}(u(B)-u(A))-\frac{|A B|}{|I L|} \beta_{L, \sigma}(u(I)-u(L))+O(h^2). 
	\end{split}
\end{align}
$\alpha_{K, \sigma}$,  $\alpha_{L, \sigma}$,  $\beta_{K, \sigma}$ and   $\beta_{L, \sigma}$ are the  integral means of $\alpha_K$,  $\alpha_{L}$, $\beta_K$ and $\beta_L$ on edge $\sigma$, respectively, 

With the help of the continuity of normal flux $\mathcal{F}_{K, \sigma}+\mathcal{F}_{L, \sigma}=0$ and omitting the $O(h^2)$ of $\mathcal{F}_{K, \sigma}$ and $\mathcal{F}_{L, \sigma}$  derives the expression of $u(I)$, and then substituting $u(I)$ into \eqref{nine-point-flux-K-1}  leads to 
\begin{align}
	\begin{split}
		&	\mathcal{F}_{K, \sigma} =\tau_{\sigma}(u(K)-u(L))+\tau_{\sigma} D_{\sigma}(u(A)-u(B)+O(h^2), \label{nine-point-flux-K}
		\\
		&	\mathcal{F}_{L, \sigma} =\tau_{\sigma}(u(L)-u(K))+\tau_{\sigma} D_{\sigma}(u(B)-u(A)+O(h^2), 
	\end{split}
\end{align}
where $\tau_{\sigma}=\frac{|A B|}{\frac{|I K|}{\beta_{K, \sigma}}+\frac{|I L|}{\beta_{L, \sigma}}}$, $D_{\sigma}=\frac{|I K| \alpha_{K, \sigma}}{|A B| \beta_{K, \sigma}}+\frac{|I L| \alpha_{L, \sigma}}{|A B| \beta_{L, \sigma}}$.
It can be seen that  $\tau_{\sigma} > 0$ since $\beta_K >0$, but $D_{\sigma}$ may be negative where its sign depends on the diffusion tensor    and   mesh cell-geometry.

Using the  second-order  method in  \cite{Sheng2008} to approximate the vertex unknowns $u(A)$ and $u(B)$,
it  follows that
\begin{align*}
	&u(A) =\sum_{i=1}^{J_A} \omega_{A_i}u({K_{A_i}}),
	\\
	&u(B) =\sum_{i=1}^{J_B} \omega_{B_i}u({K_{B_i}}),
\end{align*}
where $J_A$ (resp.  $J_B$) denotes  the number   of cells  involved in the approximation  of  $u(A)$  (resp.  $u(B)$). The  weighted coefficients  satisfy $\sum_{i=1}^{J_A} \omega_{A_i} =1$ and $\sum_{i=1}^{J_B} \omega_{B_i} =1$, and are not limited to be nonnegative.

Substituting the expressions of $u(A)$ and $u(B)$ into \eqref{nine-point-flux-K} yields  the following expressions 
\begin{align*}
	&	\mathcal{F}_{K, \sigma} =\tau_{\sigma}(u(K)-u(L))+\tau_{\sigma} D_{\sigma}\left(\sum_{i=1}^{J_A} \omega_{A_i}u({K_{A_i}})-\sum_{i=1}^{J_B} \omega_{B_i}u({K_{B_i}})\right) +O(h^2), 
	\\
	&	\mathcal{F}_{L, \sigma} =\tau_{\sigma}(u(L)-u(K))-\tau_{\sigma} D_{\sigma}\left(\sum_{i=1}^{J_A} \omega_{A_i}u({K_{A_i}})-\sum_{i=1}^{J_B} \omega_{B_i}u({K_{B_i}})\right)  +O(h^2). 
\end{align*}
Denote by  $F_{K, \sigma}^{NP}$ and $F_{L, \sigma}^{NP}$ the  numerical fluxes  of  the nine-point scheme   \cite{Sheng2008}, we have
\begin{align}
	\begin{split}
		&F_{K, \sigma}^{NP}=\tau_{\sigma}(U_K-U_L)+ \Delta_{\sigma}, \label{nine-point-flux-K-3}
		\\
		&F_{L, \sigma}^{NP} =\tau_{\sigma}(U_L-U_K)- \Delta_{\sigma},  
	\end{split}
\end{align}
where  $\displaystyle \Delta_{\sigma} = \tau_{\sigma} D_{\sigma}\left(\sum_{i=1}^{J_A} \omega_{A_i}U_{{K_{A_i}}}-\sum_{i=1}^{J_B} \omega_{B_i}U_{{K_{B_i}}}\right)$, $U$ represents the discrete solution.

Denote $U_{K_{1}}$ and $U_{K_{2}}$ such that 
\begin{align}
	\begin{split}
		&U_{K_{1}}=\min _{\bar{K} \in \mathcal{J}_{K}} U_{\bar{K}}, \label{U_K1}
		\\
		&U_{K_{2}}=\max _{\bar{K} \in \mathcal{J}_{K}} U_{\bar{K}},
	\end{split}
\end{align}
where $\mathcal{J}_{K}$ denotes  the set of cells that have    common vertices with  cell $K$ except for  $K$.

In the numerical algorithm, we set two small positive constants $\varepsilon_0$ and $\varepsilon_1$, where $\varepsilon_0, \varepsilon_1 \leq Ch^2$.  We set $\varepsilon_0=\varepsilon_1=10^{-10}$ in our numerical experiments.  

If $|\Delta_{\sigma}|\leq \varepsilon_0$,  we define   the numerical fluxes as follows:
\begin{align*}
	&F_{K, \sigma}  =\tau_{\sigma}\left(U_{K}-U_{L}\right),
	\\
	&F_{L, \sigma}  =\tau_{\sigma}\left(U_{L}-U_{K}\right).
\end{align*}
If $|\Delta_{\sigma}|> \varepsilon_0$,  then the construction contains two cases:

$\mathbf{Case~1}$.   There exist $U_{K'}$ and $U_{L'}$ such  that
\begin{align}
	\begin{split}
		&\Delta_{\sigma}(U_{K}-U_{K^{\prime}})>0, \label{case_2_1_a}
		\\
		&\Delta_{\sigma}(U_{L^{\prime}}-U_{L})>0, 
	\end{split}
\end{align}
where $U_{K'}$ and $U_{L'}$ are the cell-centered unknowns surrounding $K$ and $L$, respectively.
$U_{K'}$   and $U_{L'}$ can be taken  as $U_{K_{1}}$ or  $U_{K_{2}}$ in   \eqref{U_K1}.
In this case, the numerical fluxes  are  defined as 
\begin{align*}
	&F_{K, \sigma}  = \tau_{\sigma}\left(U_{K}-U_{L}\right) + \eta_{K,\sigma}(U_{K}-U_{K^{\prime}}), \\
	&F_{L, \sigma}  =\tau_{\sigma}\left(U_{L}-U_{K}\right)+ \eta_{L,\sigma} (U_{L}-U_{L^{\prime}}),
\end{align*}
where $\eta_{K,\sigma} = \frac{\Delta_{\sigma}}{U_{K}-U_{K^{\prime}}}$ and $\eta_{L,\sigma} = \frac{\Delta_{\sigma}}{U_{L^{\prime}}-U_{L}}$.  It can be obtained that the nonlinear coefficients   $\eta_{K,\sigma}>0$ and $\eta_{L,\sigma} >0$  since  $\Delta_{\sigma}$ and $U_{K}-U_{K^{\prime}}$, $U_{L^{\prime}}-U_{L}$  are of the same sign in this case.

$\mathbf{Case~2}$.  There do not  exist $U_{K'}$ and $U_{L'}$ such  that \eqref{case_2_1_a} hold. It is equivalent to
\begin{align}
	\Delta_{\sigma}(U_{K}-U_{K^{\prime}}) \leq  0 \label{case2-1}
\end{align}
for any $K^{\prime} \in \mathcal{J}_K$, or
\begin{align}
	\Delta_{\sigma}(U_{L^{\prime}}-U_{L}) \leq  0 \label{case2-2}
\end{align}
for any $L^{\prime} \in \mathcal{J}_L$. 
It implies that $U$ reaches its maximum or minimum on  cell $K$ or $L$, and from \eqref{case2-1} and  \eqref{case2-2} we have
\begin{align}
	\Delta_{\sigma}(U_{K}-U_{L}) \leq  0. \label{case_2_2_eq1}
\end{align}

Rewrite the linear  numerical fluxes \eqref{nine-point-flux-K-3}    as
\begin{align*}
	&F_{K, \sigma}^{NP}   = (1-\gamma_0)\tau_{\sigma}\left(U_{K}-U_{L}\right)+\Delta_{\sigma} + \gamma_0\tau_{\sigma}(U_K-U_L),
	\\
	&F_{L, \sigma} ^{NP} = (1-\gamma_0)\tau_{\sigma}\left(U_{L}-U_{K}\right)- \Delta_{\sigma}+ \gamma_0\tau_{\sigma}(U_L-U_K),
\end{align*}
where  $\gamma_0$ is a nonlinear coefficient  determined later.  To preserve the IRP, the numerical flux needs to satisfy  the DMP  structure, i.e., the numerical flux should  be  the convex combinations 
of $U_K-U_{K_j}$.   To  guarantee  $1-\gamma_0 >0$,   the  coefficient $\gamma_0$ should  satisfy
\begin{align*}
	0 \leq \gamma_0 \leq 1-\varepsilon_1.
\end{align*}

In this case, we define the final nonlinear numerical fluxes:
\begin{align*}
	&F_{K, \sigma}  = (1-\gamma_0)\tau_{\sigma}\left(U_{K}-U_{L}\right),
	\\
	&F_{L, \sigma}  = (1-\gamma_0)\tau_{\sigma}\left(U_{L}-U_{K}\right).
\end{align*}
If $|U_K-U_L|=0$, we set $\gamma_0 = 1-\varepsilon_1$, 
otherwise, we set 
$$ \gamma_{0}=\left\{
\begin{aligned}
	&\frac{-\Delta_{\sigma}}{\tau_{\sigma}\left(U_{K}-U_{L}\right)},&& \quad \text{if~}  \frac{-\Delta_{\sigma}}{\tau_{\sigma}\left(U_{K}-U_{L}\right)} \leq 1-\varepsilon_1,
	\\
	& 1-\varepsilon_1,&& \quad \text{else}.
\end{aligned}
\right.
$$
The inequality  \eqref{case_2_2_eq1} guarantees the nonnegativity of $\gamma_{0}$.

Up to now we have constructed the numerical flux with DMP-preserving structure, where the numerical flux 
is conservative and nonlinear.

	
	To discretize the  time  derivative, we utilize the Backward Euler method.  Hence, we obtain the
	nonlinear finite volume scheme \eqref{fvm-1}-\eqref{fvm-2}: 
	\begin{align}
		m(K) \frac{U_{K}^{n+1}-U_{K}^{n}}{\Delta t}+\sum_{\sigma \in \mathcal{E}_{K}}F_{K, \sigma}^{n+1}&=m(K) f_{1}(U_{K}^{n+1}, V_{K}^{n+1}) ,  & & \forall K \in \mathcal{J}_{in}, \label{fvm-1}
		\\
		m(K) \frac{V_{K}^{n+1}-V_{K}^{n}}{\Delta t}+\sum_{\sigma \in \mathcal{E}_{K}} \tilde{F}_{K, \sigma}^{n+1}&=m(K) f_{2}(U_{K}^{n+1}, V_{K}^{n+1}),  & & \forall K \in \mathcal{J}_{in},
		\\
		U_K^{n+1} &= g_1(K,t^{n+1}),  & & \forall K \in \mathcal{J}_{out},
		\\
		V_K^{n+1} &= g_2(K,t^{n+1}),   & & \forall K \in \mathcal{J}_{out},
		\\
		U_K^0 &=   u_0(K), & &  \forall K \in \mathcal{J}_{in}\cup \mathcal{J}_{out},
		\\
		V_K^0&=  v_0(K),   & &  \forall K \in \mathcal{J}_{in}\cup \mathcal{J}_{out}.   \label{fvm-2}
	\end{align}

	We  show   that the finite volume scheme \eqref{fvm-1}-\eqref{fvm-2} can preserve the IRP for the  semilinear parabolic systems with three basic  types of quasimonotone functions, which are shown in    Theorem \ref{thm1-fvm} and Theorem \ref{thm2-fvm}.
	\begin{theorem}\label{thm1-fvm}
		Assume that $\mathbf{f}=\left(f_1, f_2\right)$  is mixed quasimonotone and Lipschitz continuous in $\Sigma =[m_1,M_1]\times [m_2,M_2]$ and satisfies   \eqref{condition_invariant}, the initial and boundary  conditions   satisfy  $(u_0,v_0), (g_1,g_2) \in \Sigma$. When the time-step size satisfies   $\Delta t < \dfrac{1}{\lambda}$,
		the finite volume scheme \eqref{fvm-1}-\eqref{fvm-2} has a solution satisfying 
		\begin{eqnarray}\label{theorem_eq_1}
			(U_K^n,V_K^n)\in \Sigma, \quad \forall  K \in  \mathcal{J},\quad   0\leq  n\leq  N.
		\end{eqnarray}
	\end{theorem}
	\begin{proof}
		This theorem is proved by induction.  For  mixed quasimonotone reaction-diffusion systems, we only   prove the  theorem in   the case that  $f_1$ is quasimonotone nonincreasing and $f_2$ is quasimonotone nondecreasing, the other case   can be proved similarly.
		Since   $(u_0,v_0)$, $(g_1,g_2) \in \Sigma$, it is obvious that \eqref{theorem_eq_1} holds for $n=0$.
		Suppose that \eqref{theorem_eq_1} holds for any $n \leq m$, if \eqref{theorem_eq_1} is proved for $n = m+1$,  where $m\geq 0$, then the theorem follows immediately.
		
		First,   construct the prolongation functions $\bar f_1(u_1,u_2)$ and $\bar f_2(u_1,u_2)$ as follows. For any $u_1 \in \mathbb{R}$, define 
		$$ \bar{f}_1(u_1,u_2) = \left\{
		\begin{aligned}
			& f_1(u_1,m_2),    \quad  &&\text{if~}  u_2<m_1, \\
			& f_1(u_1,u_2),    \quad  &&\text{if~} m_2 \leq  u_2\leq  M_2, \\
			&  f_1(u_1,M_2),     \quad  &&\text{if~}   u_2> M_2.
		\end{aligned}
		\right.
		$$
		For any $u_2 \in \mathbb{R}$, define 
		$$ \bar{f}_2(u_1,u_2) = \left\{
		\begin{aligned}
			& f_2(m_1,u_2),    \quad  &&\text{if~}  u_1<m_1, \\
			& f_2(u_1,u_2),    \quad  &&\text{if~} m_1 \leq  u_1\leq  M_1, \\
			&  f_2(M_1,u_2),     \quad  &&\text{if~}   u_1> M_1.
		\end{aligned}
		\right.
		$$
		Note that $(\bar f_1, \bar f_2)$ is also mixed quasimonotone  in $\mathbb{R}^2$.  It is easy to check that $\bar f_1$ and  $\bar f_2$ are Lipschitz continuous  in $\mathbb{R}^2$ and $\lambda$ is  still the  Lipschitz constant of $\bar f_1$ and  $\bar f_2$, and it also holds that
		\begin{align*}
			\bar f_1(M_1,m_2) \leq  0 \leq  \bar  f_1(m_1,M_2),
			\\
			\bar f_2(M_1,M_2) \leq 0 \leq \bar f_2(m_1,m_2).
		\end{align*}

		Next,  for given $(U^m,V^m)\in \Sigma$, denote $(\bar{U}^{m+1},\bar{V}^{m+1})$  the solution of   following problem:
		\begin{align}
			m(K) \frac{\bar U_{K}^{m+1}-  U_{K}^{m}}{\Delta t}+  \sum_{\sigma \in \mathcal{E}_{K}}F_{K, \sigma}^{m+1}&=m(K) \bar f_{1}(\bar U_{K}^{m+1}, \bar V_{K}^{m+1}) ,  & & \forall K \in \mathcal{J}_{in}, \label{P1_1}
			\\
			m(K) \frac{\bar V_{K}^{m+1}- V_{K}^{m}}{\Delta t}+ \sum_{\sigma \in \mathcal{E}_{K}}{ \tilde{F}}_{K, \sigma}^{m+1}&=m(K)\bar  f_{2}(\bar U_{K}^{m+1}, \bar V_{K}^{m+1}), & & \forall K \in \mathcal{J}_{in},
			\\
			\bar 	U_K^{m+1} &= g_1(K,t^{m+1}),  & & \forall K \in \mathcal{J}_{out},
			\\
			\bar 	V_K^{m+1} &= g_2(K,t^{m+1}),   & & \forall K \in \mathcal{J}_{out}, \label{P1_2}
		\end{align}
		where ${F}_{K,\sigma}^{m+1}$ and $ { \tilde{F}}_{K, \sigma}^{m+1}$  are  constructed by $\bar{U}_K^{m+1}$ and $\bar{V}_K^{m+1}$.

		Denote 	 $K_{\max,U}$, $K_{\min,U}$, $K_{\max,V}$, $K_{\min,V}$ such that
		\begin{align*}
			& \bar U_{K_{\max,U}}^{m+1} = \max_{K \in  \mathcal{J}} \bar U^{m+1}_K, \quad \bar U_{K_{\min,U}}^{m+1} = \min_{K \in  \mathcal{J}} \bar U^{m+1}_K,
			\\
			& \bar V_{K_{\max,V}}^{m+1} = \max_{K \in  \mathcal{J}} \bar V^{m+1}_K, \quad \bar  V_{K_{\min,V}}^{m+1} = \min_{K \in  \mathcal{J}} \bar V^{m+1}_K.
		\end{align*}
		
		If \eqref{theorem_eq_1} does not hold for $n=m+1$, then one of the following cases holds:
		\begin{align*}
			&\text{Case~1:}\quad \bar U_{K_{\max,U}}^{m+1}  >M_1. \quad
			\text{Case~2:}\quad  \bar  U_{K_{\min,U}}^{m+1} <m_1.
			\\
			&\text{Case~3:}\quad  \bar V_{K_{\max,V}}^{m+1}  >M_2. \quad
			\text{Case~4:}\quad  \bar  V_{K_{\min,V}}^{m+1} <m_2.
		\end{align*}
		We shall prove that none of the above cases hold true by contradiction. 
		Let us suppose Case 1 holds.   Since $ \bar U^{m+1}$ attains its maximum on $K_{\max,U}$ and the numerical flux has DMP-preserving structure, it holds that  $F_{K_{\max,U}, \sigma}^{m+1} \leq 0$ for any $\sigma \in  \mathcal{E}_{K_{\max,U}}$.  From  $   U_{K_{\max,U}}^{m} \in [m_1,M_1]$  and \eqref{fvm-1}, we can get 
		\begin{align}
			\frac{ \bar U_{K_{\max,U}}^{m+1}-M_1}{\Delta t} \leq	\frac{ \bar U_{K_{\max,U}}^{m+1}-   U_{K_{\max,U}}^{m}}{\Delta t}  \leq \bar f_{1}( \bar U_{K_{\max,U}}^{m+1}, \bar V_{K_{\max,U}}^{m+1}). \label{the1_eq1}
		\end{align}

		(1)   $ \bar V_{K_{\max, U}}^{m+1}\geq m_{2}$.
		
		Since $\bar f_1$ is nonincreasing with respect to $v$, we have $\bar f_{1}(M_{1}, \bar V_{K_{\max,U}}^{m+1})\leq \bar f_{1}(M_{1}, m_{2}) \leq 0$. According to \eqref{the1_eq1} and Lipschitz continuity of   $\bar f_1$, we have 
		\begin{align}
			\begin{aligned}
				\frac{ \bar U_{K_{\max,U}}^{m+1}-M_1}{\Delta t}  \leq \bar  f_{1}( \bar U_{K_{\max,U}}^{m+1}, \bar V_{K_{\max,U}}^{m+1})  -\bar f_{1}(M_{1}, \bar V_{K_{\max,U}}^{m+1})
				\leq \lambda  ( \bar U_{K_{\max,U}}^{m+1}-M_{1}  ).  
			\end{aligned}\label{the1_eq2}
		\end{align}
		When $\Delta t <  \dfrac{1}{\lambda}$, \eqref{the1_eq2} implies that $ \bar U_{K_{\max,U}}^{m+1} \leq M_1$, which contradicts with the assumption that  $ \bar U_{K_{\max,U}}^{m+1} > M_1$.
		
		(2) $ \bar V_{K_{\max, U}}^{m+1}< m_{2}$.
		
		From $ \bar V_{K_{\max, U}}^{m+1}<m_{2}$, 
		we have $ \bar V_{K_{\min, V}}^{m+1}< m_{2}$. Since $\bar{V}^{m+1}$ attains its minimum on $K_{\min,V}$,  we can obtain
		\begin{align}
			\frac{ \bar V_{K_{\min,V}}^{m+1}-m_2}{\Delta t} \geq	\frac{ \bar V_{K_{\min,V}}^{m+1}- \bar V_{K_{\min,V}}^{m}}{\Delta t}  \geq \bar  f_{2}( \bar U_{K_{\min,V}}^{m+1}, \bar V_{K_{\min,V}}^{m+1}). \label{the1_eq3}
		\end{align}
		
		(2.1)  $ \bar U_{K_{\min, V}}^{m+1}\geq m_{1}$.
		
		Since $\bar f_2$ is nondecreasing with respect to $u$, we have $ \bar f_{2}( \bar U_{K_{\min, U}}^{m+1},m_2)\geq  \bar f_{2}(m_{1}, m_{2}) \geq 0$.  Similar to the derivation of \eqref{the1_eq2}, it follows that
		\begin{align}
			\begin{aligned}
				\frac{ \bar V_{K_{\min,V}}^{m+1}-m_2}{\Delta t}  \geq \bar  f_{2}( \bar U_{K_{\min,V}}^{m+1}, \bar V_{K_{\min,V}}^{m+1}) -\bar  f_{2}( \bar U_{K_{\min, U}}^{m+1},m_2) \geq \lambda ( \bar V_{K_{\min,V}}^{n+1}-m_{2}).
			\end{aligned} \label{the1_eq4}
		\end{align}
		When  $\Delta t <  \dfrac{1}{\lambda}$,  \eqref{the1_eq4} implies $ \bar V_{K_{\min,V}}^{m+1}\geq m_2$, which contradicts with  $ \bar V_{K_{\min, V}}^{m+1}< m_{2}$.
		
		(2.2) $ \bar U_{K_{\min, V}}^{m+1}< m_{1}$.
		
		The  assumption 	$ \bar U_{K_{\min, V}}^{m+1}< m_{1}$ implies  $ \bar U_{K_{\min, U}}^{m+1}< m_{1}$.	Since $ \bar U^{m+1}$ attains its minimum on $K_{\min,U}$,  we can obtain
		\begin{align}
			\frac{ \bar U_{K_{\min,U}}^{m+1}-m_1}{\Delta t} \geq	\frac{ \bar U_{K_{\min,U}}^{m+1}-U_{K_{\min,U}}^{m}}{\Delta t}  \geq \bar  f_{1}( \bar U_{K_{\min,U}}^{m+1}, \bar V_{K_{\min,U}}^{m+1}). \label{the1_eq5}
		\end{align}
		
		(2.2.1)	 $ \bar V_{K_{\min, U}}^{m+1}\leq M_{2}$.
		
		Since $\bar f_1$ is nonincreasing with respect to $v$, we have $\bar f_1(m_1, \bar V_{K_{\min, U}}^{m+1})\geq \bar f_1(m_1,M_2)\geq0$. 
		According to \eqref{the1_eq5}  and the Lipschitz  continuity  of $\bar f_1$, we have
		\begin{align*}
			\begin{aligned}
				\frac{ \bar U_{K_{\min,U}}^{m+1}-m_1}{\Delta t}  \geq \bar f_1( \bar U_{K_{\min, U}}^{m+1}, \bar V_{K_{\min, U}}^{m+1})  -\bar f_1(m_1, \bar V_{K_{\min, U}}^{m+1}) 
				\geq \lambda (\bar U_{K_{\min, U}}^{m+1} -m_1),
			\end{aligned}
		\end{align*}
		which means that $ \bar U_{K_{\min, U}}^{m+1} \geq m_1$ when $\Delta t < \dfrac{1}{\lambda}$. This 
		contradicts with   $ \bar U_{K_{\min, U}}^{m+1}< m_{1}$.
		
		(2.2.2) $ \bar V_{K_{\min, U}}^{m+1}> M_{2}$.
		
		The assumption 	$ \bar V_{K_{\min, U}}^{m+1}> M_{2}$ implies $ \bar V_{K_{\max, V}}^{m+1}> M_{2}$.  Since $\bar{V}^{m+1}$ attains its maximum  on $K_{\max,V}$,  we can obtain
		\begin{align}
			\frac{ \bar V_{K_{\max,V}}^{m+1}-M_2}{\Delta t} \leq	\frac{ \bar V_{K_{\max,V}}^{m+1}-V_{K_{\max,V}}^{m}}{\Delta t}  \leq \bar  f_{2}( \bar U_{K_{\max,V}}^{m+1}, \bar V_{K_{\max,V}}^{m+1}). \label{the1_eq6}
		\end{align}
		
		(2.2.2.1) $ \bar U_{K_{\max, V}}^{m+1}\leq M_{1}$.
		
		Since $\bar f_2$ is nondecreasing with respect to $u$, we have $\bar f_2( \bar U_{K_{\max, V}}^{m+1},M_2)\leq \bar f_2(M_1,M_2)\leq 0$.  According  to  \eqref{the1_eq6} and the Lipschitz  continuity of $\bar f_2$, we have
		\begin{align}
			\begin{aligned}
				\frac{ \bar V_{K_{\max,V}}^{m+1}-M_2}{\Delta t}  \leq	 \bar f_2( \bar U_{K_{\max,V}}^{m+1}, \bar V_{K_{\max,V}}^{m+1}) - \bar f_2( \bar U_{K_{\max, V}}^{m+1},M_2)
				\leq \lambda ( \bar V_{K_{\max,V}}^{m+1}-M_2), \label{the1_eq7}
			\end{aligned}
		\end{align}
		which means $ \bar V_{K_{\max,V}}^{m+1}\leq M_2$ provided that  $\Delta t < \dfrac{1}{\lambda}$. This 
		contradicts with   $ \bar V_{K_{\max, V}}^{m+1}> M_{2}$.
		
		(2.2.2.2) $ \bar U_{K_{\max, V}}^{m+1}> M_{1}$.
		
		Since $\bar f_1$ is nonincreasing with respect to $v$ and $ \bar V_{K_{\max, U}}^{m+1}>M_2>m_2$, we have $ \bar f_{1}(M_{1}, \bar V_{K_{\max,U}}^{m+1})\leq  \bar  f_{1}(M_{1}, m_{2}) \leq 0$.  
		Similar to the derivation of \eqref{the1_eq2}, it holds that 
		\begin{align*}
			\begin{aligned}
				\frac{ \bar U_{K_{\max,U}}^{m+1}-M_1}{\Delta t}    \leq \lambda  ( \bar U_{K_{\max,U}}^{m+1}-M_{1}  ),
			\end{aligned} 
		\end{align*}
		which  implies $ \bar U_{K_{\max,U}}^{m+1}\leq M_1$ provided that   $\Delta t < \dfrac{1}{\lambda}$. This 
		contradicts with   $ \bar U_{K_{\max,U}}^{m+1}>  M_1$. 
		
		In conclusion, Case 1 does not hold provided that   $\Delta t < \dfrac{1}{\lambda}$. Similarly, the other cases do not hold when  $\Delta t < \dfrac{1}{\lambda}$.  
		Hence $ (\bar U_{K}^{m+1}, \bar  V_{K}^{m+1}) \in \Sigma$ for any $K \in \mathcal{J}_{in} \cup \mathcal{J}_{out}$, it follows that 
		$\bar f_1(\bar U_{K}^{m+1}, \bar V_{K}^{m+1}) =  f_1(\bar U_{K}^{m+1}, \bar V_{K}^{m+1})$,  $\bar f_2(\bar U_{K}^{m+1}, \bar V_{K}^{m+1}) =  f_2(\bar U_{K}^{m+1}, \bar V_{K}^{m+1})$. 
		We obtain  that $ (\bar U_{K}^{m+1}, \bar  V_{K}^{m+1}) \in \Sigma$ is also the solution  of following scheme 
		\begin{align*}
			m(K) \frac{ U_{K}^{m+1}-  U_{K}^{m}}{\Delta t}+ \sum_{\sigma \in \mathcal{E}_{K}} F_{K, \sigma}^{m+1}&=m(K) f_{1}( U_{K}^{m+1},  V_{K}^{m+1}) , & & \forall K \in \mathcal{J}_{in}, 
			\\
			m(K) \frac{ V_{K}^{m+1}- V_{K}^{m}}{\Delta t}+ \sum_{\sigma \in \mathcal{E}_{K}} { \tilde{F}}_{K, \sigma}^{m+1}&=m(K)  f_{2}( U_{K}^{m+1},  V_{K}^{m+1}),  & & \forall K \in \mathcal{J}_{in},
			\\
			U_K^{m+1} &= g_1(K,t^{m+1}),  & & \forall K \in \mathcal{J}_{out},
			\\
			V_K^{m+1} &= g_2(K,t^{m+1}),   & & \forall K \in \mathcal{J}_{out},
		\end{align*}
		which means that the finite volume scheme  \eqref{fvm-1}-\eqref{fvm-2} has a solution  $  (U^{m+1},  V^{m+1})$ in the invariant region. This completes the proof.    
	\end{proof}
	
	We shall now demonstrate that the finite volume scheme can preserve the  IRP  for the semilinear parabolic systems  with quasimonotone nondecreasing  reaction function. The same method can be applied for the case of a quasimonotone nonincreasing reaction function, which we will omit for brevity. 
	\begin{theorem}\label{thm2-fvm}
		Suppose that $\mathbf{f}=\left(f_1, f_2\right)$   is quasimonotone nondecreasing (or vice versa)  and Lipschitz continuous in $\Sigma =[m_1,M_1]\times [m_2,M_2]$ and  satisfies  \eqref{condition_invariant}, the initial and boundary  conditions   satisfy  $(u_0,v_0), (g_1,g_2) \in \Sigma$. When the time-step size satisfies   $\Delta t < \dfrac{1}{2\lambda}$,
		the finite volume scheme has a solution satisfying 
		\begin{eqnarray}\label{theorem2_eq_1}
			(U_K^n,V_K^n)\in \Sigma, \quad \forall  K \in  \mathcal{J},\quad  0 \leq n\leq N.
		\end{eqnarray}
	\end{theorem}
	\begin{proof}
		We shall adopt  the same procedure  as in  the proof  of Theorem \ref{thm1-fvm}. It is obvious that \eqref{theorem2_eq_1} holds for $n=0$.
		Suppose that \eqref{theorem2_eq_1} holds for any $n \leq m$, where $m\geq 0$. Our objective is  to demonstrate that  \eqref{theorem2_eq_1} holds  for $n = m+1$.
		
		The definitions of  prolongation functions $\bar f_1(u_1,u_2)$ and $\bar f_2(u_1,u_2)$ are the  same as in the proof of Theorem \ref{thm1-fvm}, and  $(\bar f_1, \bar f_2)$ is also   quasimonotone nondecreasing   in $\mathbb{R}^2$ and satisfies
		\begin{align*}
			\bar f_1(M_1,M_2)  \leq   0 \leq  \bar  f_1(m_1,m_2),
			\\
			\bar f_2(M_1,M_2) \leq 0 \leq \bar f_2(m_1,m_2).
		\end{align*}
		Similar to the proof of Theorem  \ref{thm1-fvm},  we  define   $  (\bar{U}^{m+1},\bar{V}^{m+1})$   the solution of the finite volume  scheme corresponding  to $(\bar f_1, \bar f_2)$. The definitions of  $K_{\max,U}$, $K_{\min,U}$, $K_{\max,V}$ and $K_{\min,V}$  are also   the same as those in Theorem  \ref{thm1-fvm}. 
		
		Suppose that \eqref{theorem2_eq_1} does not hold for $n=m+1$, then one of the following cases holds:
		\begin{align*}
			&\text{Case~1:}\quad \bar U_{K_{\max,U}}^{m+1}  >M_1. \quad
			\text{Case~2:}\quad  \bar  U_{K_{\min,U}}^{m+1} <m_1.
			\\
			&\text{Case~3:}\quad  \bar V_{K_{\max,V}}^{m+1}  >M_2. \quad
			\text{Case~4:}\quad  \bar  V_{K_{\min,V}}^{m+1} <m_2.
		\end{align*}
		Let us suppose Case 1 holds.   We can obtain that   \eqref{the1_eq1} holds since $ \bar U^{m+1}$ attains its maximum on $K_{\max,U}$.           
		
		(1)  $ \bar V_{K_{\max, U}}^{m+1} \leq M_{2}$.
		
		Since 	$\bar f_1$ is nondecreasing with respect  to $v$, we have $\bar f_1(M_1, \bar V_{K_{\max, U}}^{m+1})\leq \bar f_1(M_1,M_2)\leq 0$.  Similar to the derivation of     \eqref{the1_eq2}, we obtain 
		\begin{align}
			\frac{ \bar U_{K_{\max,U}}^{m+1}-M_1}{\Delta t} 
			\leq \lambda  ( \bar U_{K_{\max,U}}^{m+1}-M_{1}  ).  \label{the2_eq2}
		\end{align}
		\eqref{the2_eq2} implies  $\bar U_{K_{\max,U}}^{m+1} \leq M_1$ provided that  $\Delta t < \dfrac{1}{\lambda}$,  this leads to  a contradiction with Case  1.
		
		(2)  $\bar V_{K_{\max, U}}^{m+1} >M_{2}$.
		
		The assumption 	$ \bar V_{K_{\max, U}}^{m+1} >M_{2}$ implies  	$ \bar V_{K_{\max, V}}^{m+1} >M_{2}$.   It  follows  that  \eqref{the1_eq6} since  $ \bar V_{K_{\max, V}}^{m+1} $ is  the maximum  of $\bar V^{m+1}$.
		
		(2.1) $\bar U_{K_{\max, V}}^{m+1} \leq  M_{1}$.
		
		Since 	$\bar f_2$ is nondecreasing with respect  to $u$, we have $\bar f_2( \bar U_{K_{\max, V}}^{m+1},M_2)\leq \bar f_2(M_1,M_2)\leq 0$.   Similar to the derivation of    \eqref{the1_eq7}, we have
		\begin{align}
			\frac{ \bar V_{K_{\max,V}}^{m+1}-M_2}{\Delta t}  \leq	  \lambda ( \bar V_{K_{\max,V}}^{m+1}-M_2). \label{the2_eq3}
		\end{align}
		When  $\Delta t < \dfrac{1}{\lambda}$,   \eqref{the2_eq3} implies that $\bar V_{K_{\max,V}}^{m+1} \leq M_2$, which contradicts with $\bar V_{K_{\max, U}}^{m+1} >M_{2}$.
		
		(2.2) $\bar U_{K_{\max, V}}^{m+1} > M_{1}$.
		
		The assumption  $\bar U_{K_{\max, V}}^{m+1} > M_{1}$ implies that  $\bar U_{K_{\max, U}}^{m+1} > M_{1}$.   From the Lipschitz  continuity of $\bar f_1$ and  $\bar f_1(M_1 ,M_2) \geq 0$, it follows that 
		\begin{align}
			\frac{ \bar U_{K_{\max,U}}^{m+1}-M_1}{\Delta t}  &\leq	\bar f_1( \bar U_{K_{\max,U}}^{m+1}, \bar V_{K_{\max,U}}^{m+1}) - \bar f_1(M_1 ,M_2)  \nonumber
			\\
			& \leq	\lambda (\bar U_{K_{\max,U}}^{m+1}-M_1) + \lambda (\bar V_{K_{\max,U}}^{m+1}-M_2). \label{the2_eq4}
		\end{align}
		From \eqref{the2_eq4}, we can  obtain that 
		\begin{align}
			\bar U_{K_{\max,U}}^{m+1}-M_1 \leq \dfrac{\lambda \Delta t}{1-\lambda \Delta t} (\bar V_{K_{\max,U}}^{m+1}-M_2). \label{the2_eq6}
		\end{align}
		Similar to the derivations of    \eqref{the2_eq4} and \eqref{the2_eq6},  $\bar V_{K_{\max,V}}^{m+1}$ satisfies the following inequalities
		\begin{align*}
			\frac{ \bar V_{K_{\max,U}}^{m+1}-M_2}{\Delta t} & \leq	\bar f_2( \bar U_{K_{\max,V}}^{m+1}, \bar V_{K_{\max,V}}^{m+1}) - \bar f_2(M_1 ,M_2) \nonumber
			\\
			&	\leq	\lambda (\bar U_{K_{\max,V}}^{m+1}-M_1) + \lambda (\bar V_{K_{\max,V}}^{m+1}-M_2),
		\end{align*}
		and
		\begin{align}
			\bar U_{K_{\max,V}}^{m+1}-M_1 \geq \dfrac{1-\lambda \Delta t}{\lambda \Delta t }(\bar V_{K_{\max,V}}^{m+1}-M_2).  \label{the2_eq7}
		\end{align}
		Combining \eqref{the2_eq6} and \eqref{the2_eq7}, we see that $\dfrac{\lambda \Delta t}{1-\lambda \Delta t} (\bar V_{K_{\max,U}}^{m+1}-M_2) \geq  \dfrac{1-\lambda \Delta t}{\lambda \Delta t }(\bar V_{K_{\max,V}}^{m+1}-M_2)$.   We have 
		$\dfrac{\lambda \Delta t}{1-\lambda \Delta t} > \dfrac{1-\lambda \Delta t}{\lambda \Delta t }$ provided that $\Delta t < \dfrac{1}{2\lambda}$, which derives  that
		$\bar V_{K_{\max,U}}^{m+1}\leq  M_2$.  It  contradicts with the assumption  that $\bar V_{K_{\max, U}}^{m+1} >M_{2}$.   Hence, Case 1 does not hold when $\Delta t < \dfrac{1}{2\lambda}$. The proofs of other cases follow similarly.
		
		In conclusion, we have proved that  $ (\bar U_{K}^{m+1}, \bar  V_{K}^{m+1})\in \Sigma$ for any $K \in \mathcal{J}_{in} \cup \mathcal{J}_{out}$, then we have 
		$\bar f_1(\bar U_{K}^{m+1}, \bar V_{K}^{m+1}) =  f_1(\bar U_{K}^{m+1}, \bar V_{K}^{m+1})$ and $\bar f_2(\bar U_{K}^{m+1}, \bar V_{K}^{m+1}) =  f_2(\bar U_{K}^{m+1}, \bar V_{K}^{m+1})$. 
		It means that $ (\bar U^{m+1}, \bar  V^{m+1}) $ is also the solution  of  the finite volume scheme  \eqref{fvm-1}-\eqref{fvm-2}, which  completes the proof.  
	\end{proof}
	
	The existence of solution for the nonlinear finite volume scheme  \eqref{fvm-1}-\eqref{fvm-2} can be established using the same method as Theorem 4 in  \cite{Zhou2022},  and so is omitted.
	\begin{theorem}\label{thm-fvm-existence}
		Suppose that $\mathbf{f}=\left(f_1, f_2\right)$  is  quasimonotone  and Lipschitz continuous in $\Sigma =[m_1,M_1]\times [m_2,M_2]$, and satisfies  \eqref{condition_invariant}, the initial and boundary  conditions    satisfy   $(u_0,v_0), (g_1,g_2) \in \Sigma$. When  $\Delta t < \dfrac{1}{2\lambda}$,    the nonlinear finite volume scheme   \eqref{fvm-1}-\eqref{fvm-2} has at least one solution.
	\end{theorem}
	
\section{The iterative  method preserving the IRP}
In this section, we  design iterative  method  to solve the nonlinear scheme   \eqref{fvm-1}-\eqref{fvm-2}, and then prove the IRP  of the iteration.
To design the iterative scheme, the nonlinear numerical flux   needs to be linearized first.  Denote $U^{n+1,s+1}$  and $V^{n+1,s+1}$  the $(s+1)$-th iterative numerical solutions at $t^{n+1}$.  We use the solution  $U^{n+1,s}$ to calculate the coefficients $\eta_{K,\sigma}^{n+1,s}$, $\eta_{L,\sigma}^{n+1,s}$, $\gamma_0^{n+1,s}$ and  $\Delta_{\sigma}^{n+1,s}$.
For Case 1 in the algorithm, the numerical flux $F_{K,\sigma}^{n+1,s+1}$ and $F_{L,\sigma}^{n+1,s+1}$    are defined as
\begin{align*}
	&F_{K, \sigma}^{n+1,s+1}  = \tau_{\sigma}^{n+1}\left(U_{K}^{n+1,s+1}-U_{L}^{n+1,s+1}\right) + \eta_{K,\sigma}^{n+1,s}\left(U_{K}^{n+1,s+1}-U_{K^{\prime}}^{n+1,s+1}\right), 
	\\
	&F_{L, \sigma}^{n+1,s+1}  =\tau_{\sigma}^{n+1}\left(U_{L}^{n+1,s+1}-U_{K}^{n+1,s+1}\right)+ \eta_{L,\sigma} ^{n+1,s}\left(U_{L}^{n+1,s+1}-U_{L^{\prime}}^{n+1,s+1}\right),
\end{align*}
and for Case 2, $F_{K,\sigma}^{n+1,s+1}$  and $F_{L,\sigma}^{n+1,s+1}$   are defined as
\begin{align*}
	&F_{K, \sigma}^{n+1,s+1}  = (1-\gamma_0^{n+1,s})\tau_{\sigma}^{n+1}\left(U_{K}^{n+1,s+1}-U_{L}^{n+1,s+1}\right),
	\\
	&F_{L, \sigma}^{n+1,s+1}  = (1-\gamma_0^{n+1,s})\tau_{\sigma}^{n+1}\left(U_{L}^{n+1,s+1}-U_{K}^{n+1,s+1}\right).
\end{align*}
$\tilde{F}_{K,\sigma}^{n+1,s+1}$  and $\tilde{F}_{L,\sigma}^{n+1,s+1}$ are defined similarly.

The treatment of the nonlinear source term is crucial for preserving IRP during the iteration.   For given $U^{n+1,s}$ and $V^{n+1,s}$,  the solution of the  iteration $(U^{n+1,s+1},V^{n+1,s+1})$ satisfies 
\begin{align}
	m(K) \dfrac{U_{K}^{n+1,s+1}-U_{K}^{n}}{\Delta t}+\sum_{\sigma \in \mathcal{E}_{K}}F_{K, \sigma}^{n+1,s+1}+\lambda m(K)U_K^{n+1,s+1}&=m(K)\left(\lambda U_K^{n+1,s}+f_{1}(U_{K}^{n+1,s}, V_{K}^{n+1,s}) \right),  \label{iter-1}
	\\
	m(K) \dfrac{V_{K}^{n+1,s+1}-V_{K}^{n}}{\Delta t}+\sum_{\sigma \in \mathcal{E}_{K}}\tilde{F}_{K, \sigma}^{n+1,s+1}+\lambda m(K)V_K^{n+1,s+1}&= m(K)\left( \lambda V_K^{n+1,s}+  f_{2}(U_{K}^{n+1,s+1}, V_{K}^{n+1,s})\right),   \label{iter-2}
\end{align}
for any $K \in \mathcal{J}_{in}$, and  subject  to $U_K^{n+1,s+1} = g_1(K,t^{n+1})$ and  $
V_K^{n+1,s+1}  = g_2(K,t^{n+1})$ for any  $K \in \mathcal{J}_{out}$,  $U_K^0  =   u_0(K)$ and  $V_K^0 =  v_0(K)$ for any $ K \in \mathcal{J}_{in}\cup \mathcal{J}_{out}$, where   $n\geq 0$,  $s\geq 0$.  The iterative algorithm is described in Algorithm \ref{Algorithm_1}, where  we set $\varepsilon = 10^{-8}$ in the numerical  experiments.

\begin{algorithm}[!htbp]
	\caption{The IRP-preserving iteration}\label{Algorithm_1}
	\label{2}
	\begin{algorithmic}[1]
		\State Compute the initial  vector  $(U^0,V^0)$;
		\State $n = 0$;
		\While{$t^{n+1} \leq T$}
		\State Let $s=0$;
		\State Take $(U^{n+1,0},V^{n+1,0}) = (U^{n},V^{n}) $;
		\While{$||U^{n+1,s+1} - U^{n+1,s}||_{\infty} > \varepsilon_{non}$ or $||V^{n+1,s+1} - V^{n+1,s}||_{\infty} > \varepsilon_{non}$}
		\State  Solve the  linear system \eqref{iter-1}-\eqref{iter-2};
	\State  Let $s=s+1$;
	\EndWhile
	\State  Let $n = n+1$;
	\EndWhile
\end{algorithmic}
\end{algorithm}

We prove that the iterative method can preserve the IRP for the coupled quasimonotone  parabolic system.
\begin{theorem}
Suppose that $\mathbf{f}=\left(f_1, f_2\right)$ is quasimonotone and Lipschitz continuous in $\Sigma =[m_1,M_1]\times [m_2,M_2]$ and  satisfies  \eqref{condition_invariant}, the initial and boundary  conditions   satisfy  $(u_0,v_0) \in \Sigma$,  $(g_1,g_2) \in \Sigma$, then for any  $\Delta  t>0$, the solution of the  iteration \eqref{iter-1}-\eqref{iter-2} satisfies
\begin{eqnarray}\label{theorem_2_eq}
	(U_K^{n,s}, V_K^{n,s} ) \in \Sigma, \quad \forall K \in  \mathcal{J}, \quad  0\leq n\leq N, \quad s\geq0.
\end{eqnarray}
\end{theorem}

\begin{proof}
The  proof of this theorem is   similar to that of Theorem  \ref{thm1-fvm}.    
We assume that the reaction function is mixed quasimonotone, where $f_1$ is quasimonotone nonincreasing and $f_2$  is quasimonotone nondecreasing. The proofs  for    other types of quasimonotone reaction functions are  similar.
It is easy to see that $(U^0_K,V^0_K) \in \Sigma$ for any  $K \in  \mathcal{J}_{in}  \cup \mathcal{J}_{out}$.  Suppose that \eqref{theorem_2_eq} holds for any $n \leq  m+1$, $s\leq s_0$, where $m\geq 0$,  $s_0\geq 0$. If \eqref{theorem_2_eq} is proved for $n = m+1$, $s= s_0+1$, then the theorem follows immediately.

Denote  $K_{\max,U}$, $K_{\min,U}$, $K_{\max,V}$, $K_{\min,V}$  such that
\begin{eqnarray*}
	&&U_{K_{\max,U}}^{m+1,s_0+1} = \max_{K \in  \mathcal{J}}U^{m+1,s_0+1}_K, \quad U_{K_{\min,U}}^{m+1,s_0+1} = \min_{K \in  \mathcal{J}}U^{m+1,s_0+1}_K,
	\\
	&&V_{K_{\max,V}}^{m+1,s_0+1} = \max_{K \in  \mathcal{J}}V^{m+1,s_0+1}_K, \quad V_{K_{\min,V}}^{m+1,s_0+1} = \min_{K \in  \mathcal{J}}V^{m+1,s_0+1}_K.
\end{eqnarray*}

	
	The proof can be divided into two steps.  The  first step is  to prove that  $U^{m+1,s_0+1}_{K_{\max,U}}\leq M_1$.
	
	Let us suppose that $U^{m+1,s_0+1}_{K_{\max,U}}> M_1$    holds.  Starting from  \eqref{iter-1} and   using a derivation similar to \eqref{the1_eq1}, it follows that 
	\begin{align}
		\begin{aligned}
			\frac{U^{m+1,s_0+1}_{K_{\max,U}}-M_1}{\Delta t} \leq  &	\frac{U^{m+1,s_0+1}_{K_{\max,U}}-U^{m}_{K_{\max,U}}}{\Delta t}   
			\\
			\leq  &	\lambda (U^{m+1,s_0}_{K_{\max,U}}-U^{m+1,s_0+1}_{K_{\max,U}})+f_{1}(U^{m+1,s_0}_{K_{\max,U}},V^{m+1,s_0}_{K_{\max,U}}).
		\end{aligned}\label{the2_eq1}
	\end{align}
	Since $f_1$  is nonincreasing with respect to $v$, we have $f_1(U^{m+1,s_0}_{K_{\max,U}},V^{m+1,s_0}_{K_{\max,U}})\leq f_1(U^{m+1,s_0}_{K_{\max,U}},m_2)$. Similar to the derivation of \eqref{the1_eq2}, it holds that 
	\begin{align}
		\begin{aligned}
			\frac{U^{m+1,s_0+1}_{K_{\max,U}}-M_1}{\Delta t} &\leq  \lambda (U^{m+1,s_0}_{K_{\max,U}}-U^{m+1,s_0+1}_{K_{\max,U}})+f_{1}(U^{m+1,s_0}_{K_{\max,U}},m_2)-f_1(M_1,m_2)
			\\
			& \leq  \lambda(U^{m+1,s_0}_{K_{\max,U}}-M_1) +\lambda (M_1-U^{m+1,s_0}_{K_{\max,U}} )
			=  0, \label{the_eq10}
		\end{aligned}
	\end{align}
	which contradicts with   $U_{K_{\max,U}}^{m+1,s_0+1}>  M_1$. Hence we have $U_{K_{\max,U}}^{m+1,s_0+1}\leq M_1$. $U_{K_{\min,U}}^{m+1,s_0+1}\geq m_1$ can be proved similarly. It means that $U^{m+1,s_0+1}_K\in [m_1,M_1]$ for any  $K \in  \mathcal{J}$.
	
	The second step  is to show that  $V^{m+1,s_0+1}\in [m_2,M_2]$.
	
	If $V^{m+1,s_0+1}\in [m_2,M_2]$ does not hold, then we have   $V^{m+1,s_0+1}_{K_{\max,V}}> M_2$ or  $V^{m+1,s_0+1}_{K_{\min,V}}< m_2$.  Let us suppose that   $V^{m+1,s_0+1}_{K_{\max,V}}> M_2$  holds. Similar to \eqref{the2_eq1}, it follows that 
	\begin{align}
		\begin{aligned}
			\frac{V^{m+1,s_0+1}_{K_{\max,V}}-M_2}{\Delta t} \leq   \lambda (V^{m+1,s_0}_{K_{\max,V}}-V^{m+1,s_0+1}_{K_{\max,V}})+f_{2}(U^{m+1,s_0+1}_{K_{\max,V}},V^{m+1,s_0}_{K_{\max,V}}).
		\end{aligned}
	\end{align}
	According to the monotonicity of $f_2$ with respect to $u$, we have $f_{2}(U^{m+1,s_0+1}_{K_{\max,V}},V^{m+1,s_0}_{K_{\max,V}}) \leq f_2(M_1,V^{m+1,s_0}_{K_{\max,V}})$.
	Similar to the derivation of \eqref{the_eq10}, it can be shown that
	\begin{align*}
		\begin{aligned}
			\frac{V^{m+1,s_0+1}_{K_{\max,V}}-M_2}{\Delta t} &\leq  0,
		\end{aligned}
	\end{align*}
	which contradicts with   $V_{K_{\max,V}}^{m+1,s_0+1}>  M_2$. Hence we have $V_{K_{\max,V}}^{m+1,s_0+1} \leq  M_2$. Similarly, we can prove that $V_{K_{\min,V}}^{m+1,s_0+1} \geq m_2$. It means that $V^{m+1,s_0+1}_K\in [m_2,M_2]$ for any  $K \in  \mathcal{J}$. This completes the proof.   
\end{proof}

\section{Numerical  experiments}
In this section, numerical examples with different models  are presented to show the accuracy and 
the IRP-preserving property of the finite volume scheme.  The comparison results  with the nine-point scheme are shown  to demonstrate that the nine-point scheme  fails to  preserve the  invariant regions, which indicates the advantage of the IRP-preserving scheme. In the numerical examples, we use
\begin{align*}
	&\varepsilon^2_u=\left[\sum_{K \in \mathcal{J}_{i n}}\left(U_K-u(K)\right)^2 m(K)\right]^{1 / 2}, && \varepsilon^2_v=\left[\sum_{K \in \mathcal{J}_{i n}}\left(V_K-v(K)\right)^2 m(K)\right]^{1 / 2},
	\\
	&	\varepsilon_u^F=\left[\sum_{K \in \mathcal{J}}\left(F_{K, \sigma}-\mathcal{F}_{K, \sigma}\right)^2\right]^{1 / 2}, &&\varepsilon_v^F=\left[\sum_{K \in \mathcal{J}}\left(\tilde{F}_{K, \sigma}-\tilde{\mathcal{F}}_{K, \sigma}\right)^2\right]^{1 / 2}
\end{align*}
to evaluate approximate the  $L^2$ errors  and  the normal flux  errors of $u$ and $v$, respectively.

\subsection{Example 1}
In  the first example, we consider    the problem with  continuous diffusion coefficients on $\Omega =(0,1)\times (0,1)$.   The coefficients $\kappa_1=R_1 D_1 R_1^{\mathrm{T}}$ and $\kappa_2=R_2 D_2 R_2^{\mathrm{T}}$, where $R_1$, $R_2$, $D_1$ and $D_2$ are given by 
\begin{align*}
	R_1=\left(\begin{array}{cc}
		\cos \theta_1 & -\sin \theta_1 \\
		\sin \theta_1 & \cos \theta_1
	\end{array}\right), \quad D_1=\left(\begin{array}{cc}
		k_1 & 0 \\
		0 & k_2
	\end{array}\right)
\end{align*}
and
\begin{align*}
	R_2=\left(\begin{array}{cc}
		\cos \theta_2 & -\sin \theta_2 \\
		\sin \theta_2 & \cos \theta_2
	\end{array}\right), \quad D_2=\left(\begin{array}{cc}
		k_3 & 0 \\
		0 & k_4
	\end{array}\right).
\end{align*}
We take $\theta_1=\frac{5 \pi}{12}$,  $\theta_2=\frac{\pi}{3}$,  $k_1=1+2 x^2+y^2$, $k_2=1+x^2+2 y^2$,  $k_3=1+x^2+$ $2 y^2$, $k_4=1+2 x^2+y^2$.
We set $f_1(u,v) = u(1-u)(u-0.1)$ and $f_2(u,v) = u-v$ and take the exact solution
\begin{align*}
	\begin{aligned}
		& u(x,y,t)=\mathrm{e}^{-t} \sin (\pi x) \sin (\pi y), \\
		& v(x,y,t)=\mathrm{e}^{-t} \cos (\pi x) \cos (\pi y).
	\end{aligned}
\end{align*}
The exact solution provides the Dirichlet boundary conditions. We add  linear source terms in  the right sides  of the model \eqref{problem_1}-\eqref{problem_2}, which can be calculated from  the exact solution   correspondingly.

We set the final time  $T=1$ and a sufficiently small time step $\Delta t=1$E-4.  This example only   tests the spacial  convergence order,  and no investigation on IRP was conducted.  Denote by $N_c$ the number of cells.  
We display the $L^2$  errors  and   flux errors on the random quadrilateral and  triangular meshes  in Tables \ref{exp1-quad}-\ref{exp1-tri}, respectively. 
The two types of  random meshes are presented in Figs \ref{exp1-1}-\ref{exp1-2}.
From these tables, we observe that the $L^2$ errors     obtain   second-order convergence rate   on both the random quadrilateral and triangular  meshes,    and the flux errors  obtain first-order convergence rate  on the random quadrilateral meshes and obtain higher than  first-order convergence rate on random triangular meshes.

\begin{figure}[!h]
	\begin{minipage}[t]{0.45\linewidth}
		\centering
		\includegraphics[width=2.2in]{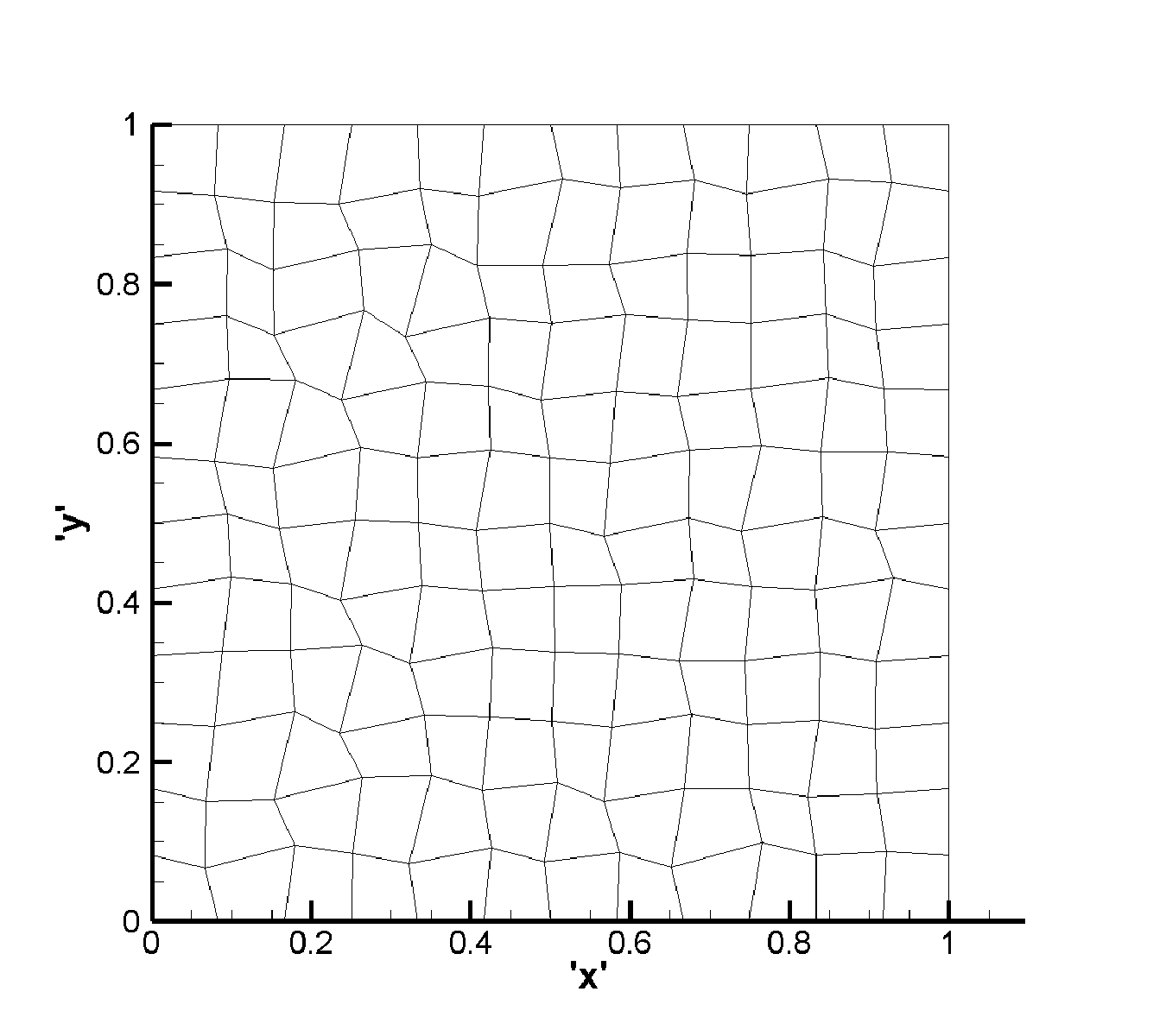}
		\caption{The  random quadrilateral  meshes.}  \label{exp1-1}
	\end{minipage}
	\hspace{1em}
	\begin{minipage}[t]{0.45\linewidth}
		\centering
		\includegraphics[width=2.2in]{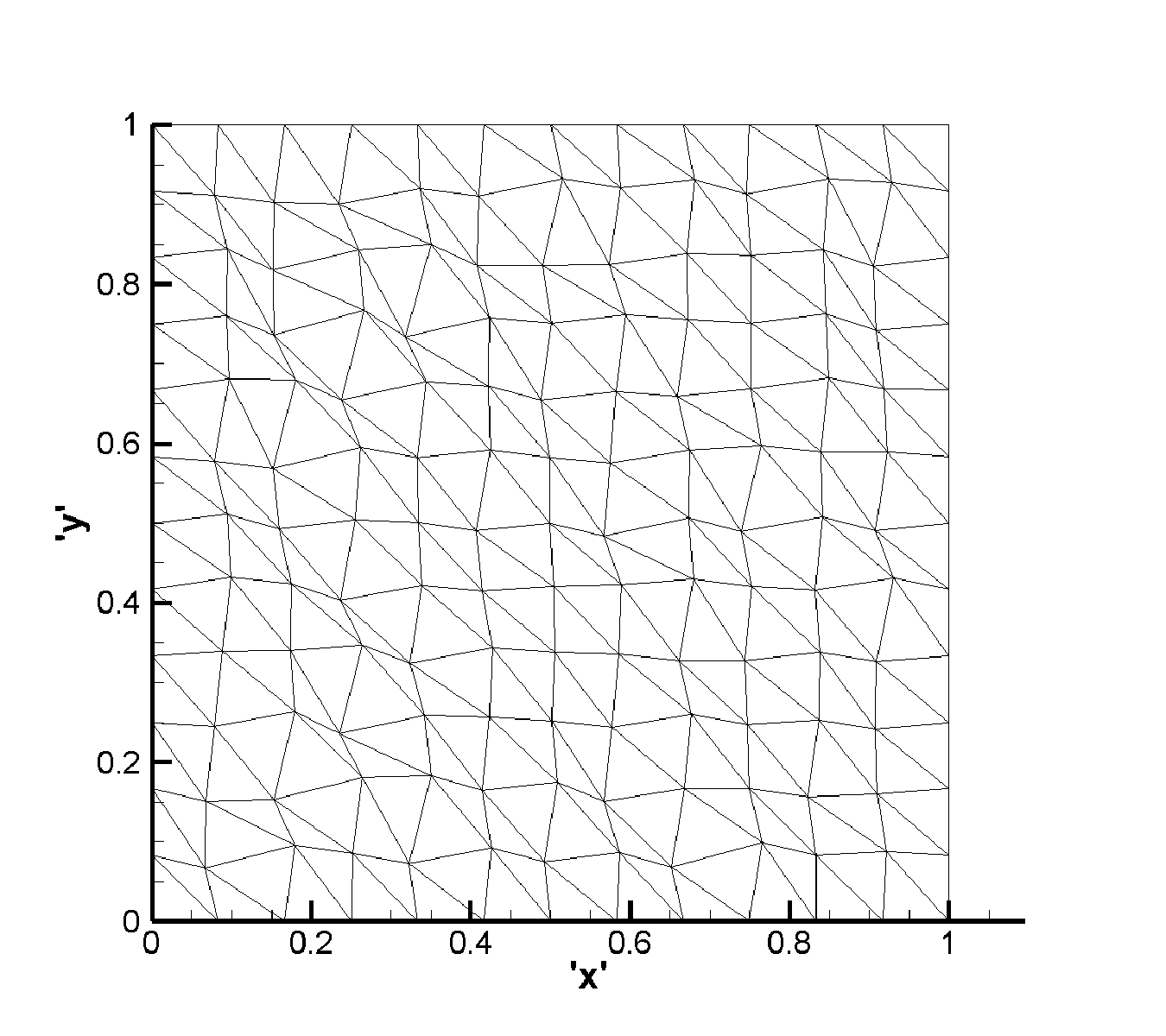}
		\caption{The  random triangular  meshes.}  \label{exp1-2}
	\end{minipage}
\end{figure}

\begin{table}[!h]
	\centering
	\caption{The   errors for Example 1 on the random quadrilateral meshes.}
	\begin{tabular}{lrrrr}
		\hline
		$N_c$ & 144   & 576   & 2304  & 9216 \\
		\hline
		$\varepsilon_u^2$         & 1.21E-03 & 3.98E-04 & 1.04E-04 & 2.21E-05 \\
		order &       & 1.61  & 1.93  & 2.23  \\
		$\varepsilon_u^F$  & 1.80E-02 & 8.27E-03 & 3.90E-03 & 2.02E-03 \\
		order &       & 1.12  & 1.08  & 0.94  \\
		$\varepsilon_v^2$ & 1.60E-03 & 4.14E-04 & 1.05E-04 & 2.63E-05 \\
		order &       & 1.94  & 1.98  & 1.99  \\
		$\varepsilon_v^F$ & 1.95E-02 & 8.50E-03 & 4.00E-03 & 1.96E-03 \\
		order &       & 1.19  & 1.08  & 1.02  \\
		\hline
	\end{tabular}%
	\label{exp1-quad}
\end{table}%
\begin{table}[!h]
	\centering
	\caption{The   errors for Example 1 on the random triangular  meshes.}
	\begin{tabular}{lrrrr}
		\hline
		$N_c$ & 288   & 1152  & 4608  & 18432 \\
		\hline
		$\varepsilon_u^2$         & 9.32E-04 & 2.45E-04 & 5.96E-05 & 1.38E-05 \\
		order &       & 1.92  & 2.03  & 2.10  \\
		$\varepsilon_u^F$  & 1.98E-02 & 6.23E-03 & 2.16E-03 & 9.31E-04 \\
		order &       & 1.66  & 1.52  & 1.21  \\
		$\varepsilon_v^2$ & 8.97E-04 & 2.52E-04 & 6.31E-05 & 1.56E-05 \\
		order &       & 1.83  & 1.99  & 2.01  \\
		$\varepsilon_v^F$ & 1.87E-02 & 6.65E-03 & 2.69E-03 & 1.11E-03 \\
		order &       & 1.49  & 1.30  & 1.27  \\
		\hline
	\end{tabular}%
	\label{exp1-tri}
\end{table}%

\subsection{Example 2}
In  this example, we test the accuracy of our scheme  under   the problem with discontinuous diffusion coefficient on $\Omega =(0,1)\times (0,1)$.  We consider the exact solution  
\begin{align*}
	\begin{aligned}
		& u(x,y,t)= \begin{cases} \mathrm{e}^{-t}\left(x-\frac{2}{3}\right)\left(x^3+y^3\right), & \text{if~} x \leq \frac{2}{3}, \\
			4\mathrm{e}^{-t}\left(x-\frac{2}{3} \right)\left(x^3+y^3\right), &  \text{if~} x>\frac{2}{3},\end{cases} \\
		& v(x,y,t)= \begin{cases}\mathrm{e}^{-t}\left(x-\frac{2}{3}\right)\left(x^2-y^2\right), & \text{if~} x \leq \frac{2}{3}, \\
			4\mathrm{e}^{-t}\left(x-\frac{2}{3}\right)\left(x^2-y^2\right), &  \text{if~}x>\frac{2}{3},\end{cases}
	\end{aligned}
\end{align*}  
and take the diffusion coefficient
\begin{align*}
	\begin{aligned}
		& \kappa_1=\kappa_2= \begin{cases}4I, & \text{if~} x \leq \frac{2}{3}, \\
			I, & \text{if~} x>\frac{2}{3}.\end{cases} \\
	\end{aligned}
\end{align*}  
The functions $f_1(u,v)$ and $f_2(u,v)$  are  
\begin{align*}
	&	f_1(u,v)=u^2-{\rm{e}}^v, \\
	&	f_2(u,v)=u^3-  v,
\end{align*}
and the linear source functions   are  calculated correspondingly.

We set the final time  $T=1$ and a sufficiently small time step $\Delta t=1$E-4 in order to test the spacial  accuracy,   no investigation on IRP was conducted.  Tables \ref{exp2-quad}-\ref{exp2-tri} provide the $L^2$ norm of errors in solutions and fluxes on two types of random meshes. The numerical  results reveal that our scheme   obtains  second-order convergence rate for the $L^2$ errors of solutions   and  is higher than first-order convergence rate for the flux  errors.

\begin{table}[htbp]  \centering  \caption{The   errors for Example 2 on the random quadrilateral meshes.}    \begin{tabular}{lrrrr}         
		\hline 
		$N_c$ & 144   & 576   & 2304  & 9216 \\
		\hline 
		\multicolumn{1}{l}{ $\varepsilon_u^2$        } & 3.44E-03 & 8.96E-04 & 2.26E-04 & 5.67E-05 
		\\
		order&       & 1.94  & 1.98  & 1.99  
		\\
		\multicolumn{1}{l}{$\varepsilon_u^F$ } & 2.14E-02 & 7.63E-03 & 3.03E-03 & 1.43E-03 
		\\
		order&       & 1.48  & 1.33  & 1.08  
		\\
		\multicolumn{1}{l}{$\varepsilon_v^2$} & 2.26E-03 & 5.77E-04 & 1.45E-04 & 3.66E-05 
		\\
		order&       & 1.96  & 1.99  & 1.98  
		\\
		\multicolumn{1}{l}{$\varepsilon_v^F$} & 1.62E-02 & 5.84E-03 & 2.38E-03 & 1.11E-03 
		\\
		order	&       & 1.47  & 1.29  & 1.10  
		
		\\
		\hline
	\end{tabular}   \label{exp2-quad}
\end{table}%

\begin{table}[htbp]
	\centering
	\caption{The   errors for Example 2 on the random triangular  meshes.}
	\begin{tabular}{lrrrr}
		\hline
		$N_c$ & 288   & 1152  & 4608  & 18432 \\
		\hline
		$\varepsilon_u^2$         & 2.02E-03 & 5.44E-04 & 1.36E-04 & 3.42E-05 \\
		order &       & 1.89  & 1.99  & 1.99  \\
		$\varepsilon_u^F$  & 1.41E-02 & 4.61E-03 & 1.53E-03 & 6.10E-04 \\
		order &       & 1.61  & 1.59  & 1.32  \\
		$\varepsilon_v^2$ & 7.40E-04 & 2.00E-04 & 4.94E-05 & 1.27E-05 \\
		order &       & 1.88  & 2.01  & 1.95  \\
		$\varepsilon_v^F$ & 1.00E-02 & 3.24E-03 & 1.09E-03 & 4.00E-04 \\
		order &       & 1.62  & 1.56  & 1.45  \\
		\hline
	\end{tabular}%
	\label{exp2-tri}
\end{table}%

\subsection{Example 3}
In Example 3, we  consider a  semilinear parabolic system illustrating  the  superconductivity of liquids, where $\kappa_1$ and $\kappa_2$ are positive definite diagonal matrices, and the nonlinear reaction-diffusion terms are as follows:
\begin{align*}
	& f_1(u,v)  = (1-u^2-v^2)u,
	\\
	&f_2(u,v)  = (1-u^2-v^2)v.
\end{align*}
$(f_1,f_2)$ is quasimonotone nonincreasing in $[0,+\infty)\times [0,+\infty)$.
The domain of this example is set to be a square with a hole $\Omega  = (0,1)^2\backslash [4/9,5/9]^2$, where internal and external boundaries are denoted by $\Gamma_ 1$ and $\Gamma_2$, respectively.  Let  the diffusion coefficients $\kappa_1$ and $\kappa_2$ be  discontinuous at   $\Gamma$, where $\Gamma$ are composed of the edges of the square $(2/9,7/9)^2$.  The domain $\Omega$   is divided into two parts by $\Gamma$,  where $\Omega_1 = (0,1)^2 \backslash [2/9,7/9]^2$ and $\Omega_2= \Omega \backslash  \Omega_1$.  We take the diffusion coefficients  
\begin{align*}
	\begin{aligned}
		& \kappa_1= \begin{cases} 2I,  &\text{in~} \Omega_1,
			\\
			\left(\begin{array}{cc}
				y^2+\varepsilon & -(1-\varepsilon) x y \\
				-(1-\varepsilon) x y & x^2+\varepsilon
			\end{array}
			\right),  & \text{in~} \Omega_2,
		\end{cases} \\
		&\kappa_2= \begin{cases}  I,  &\text{in~} \Omega_1,
			\\
			\left(\begin{array}{cc}
				y^2+\varepsilon &  (1-\varepsilon) x y \\
				(1-\varepsilon) x y & x^2+\varepsilon
			\end{array}
			\right),  & \text{in~} \Omega_2,
		\end{cases}
	\end{aligned}
\end{align*}  
where $\varepsilon = 5$E-3,  and take  the  initial and   boundary conditions as 
\begin{align*}
	\begin{aligned}
		g_1(x,y,t)=\begin{cases} 
			0, & \text { in  } \Gamma_1,
			\\
			1, & \text { in  } \Gamma_2,
		\end{cases}  
	\end{aligned}
	\quad 
	\begin{aligned}
		g_2(x,y,t)=\begin{cases} 
			1, & \text { in  } \Gamma_1,
			\\
			0, & \text { in  } \Gamma_2,
		\end{cases}  
	\end{aligned}
\end{align*}
and
\begin{align*}
	\begin{aligned}
		u_0(x,y)=\begin{cases} 
			0, & \text { in  } \Omega_1,
			\\
			1, & \text { in  } \Omega_2,
		\end{cases}  
	\end{aligned}
	\quad 
	\begin{aligned}
		v_0(x,y)=\begin{cases} 
			1, & \text { in  } \Omega_1,
			\\
			0, & \text { in  } \Omega_2.
		\end{cases}  
	\end{aligned}
\end{align*}
It is obvious  that $[0,1]\times [0,1]$ is its invariant region, and the Lipschitz constant $\lambda  = 5$.

We observe the invariant region properties on random quadrilateral meshes with $N_c = 2916$ and random triangular meshes with $N_c = 5832$, respectively.
The numerical solutions are displayed in Figs \ref{exp3-1}-\ref{exp3-2} for random quadrilateral meshes , and in Figs \ref{exp3-3}-\ref{exp3-4} for random triangular meshes , respectively.
The minimum and maximum values  for $U$ on  random quadrilateral  are 0 and 1, respectively. Similarly, the minimum and maximum values  for $V$ on   random  triangular meshes are 0 and 1. The numerical results demonstrate the IRPs of our scheme.

\begin{figure}[!h]
	\begin{minipage}[t]{0.45\linewidth}
		\centering
		\includegraphics[width=2.2in]{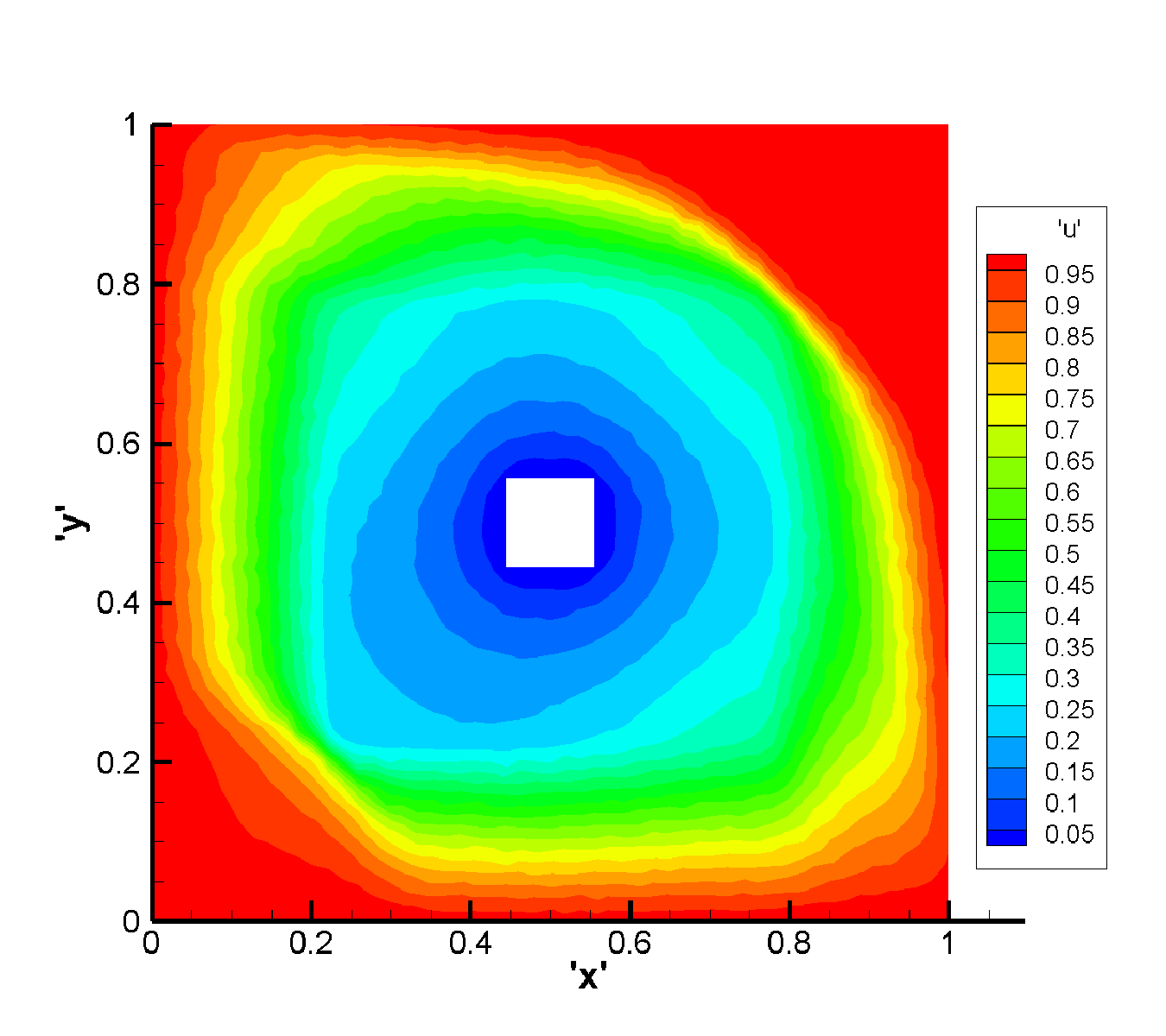}
		\caption{The numerical solution $U$ of the IRP-preserving scheme  for Example 3 on the random quadrilateral meshes ($U_{\min} =  0$, $U_{\max} =   1$).}  \label{exp3-1}
	\end{minipage}
	\hspace{1em}
	\begin{minipage}[t]{0.45\linewidth}
		\centering
		\includegraphics[width=2.2in]{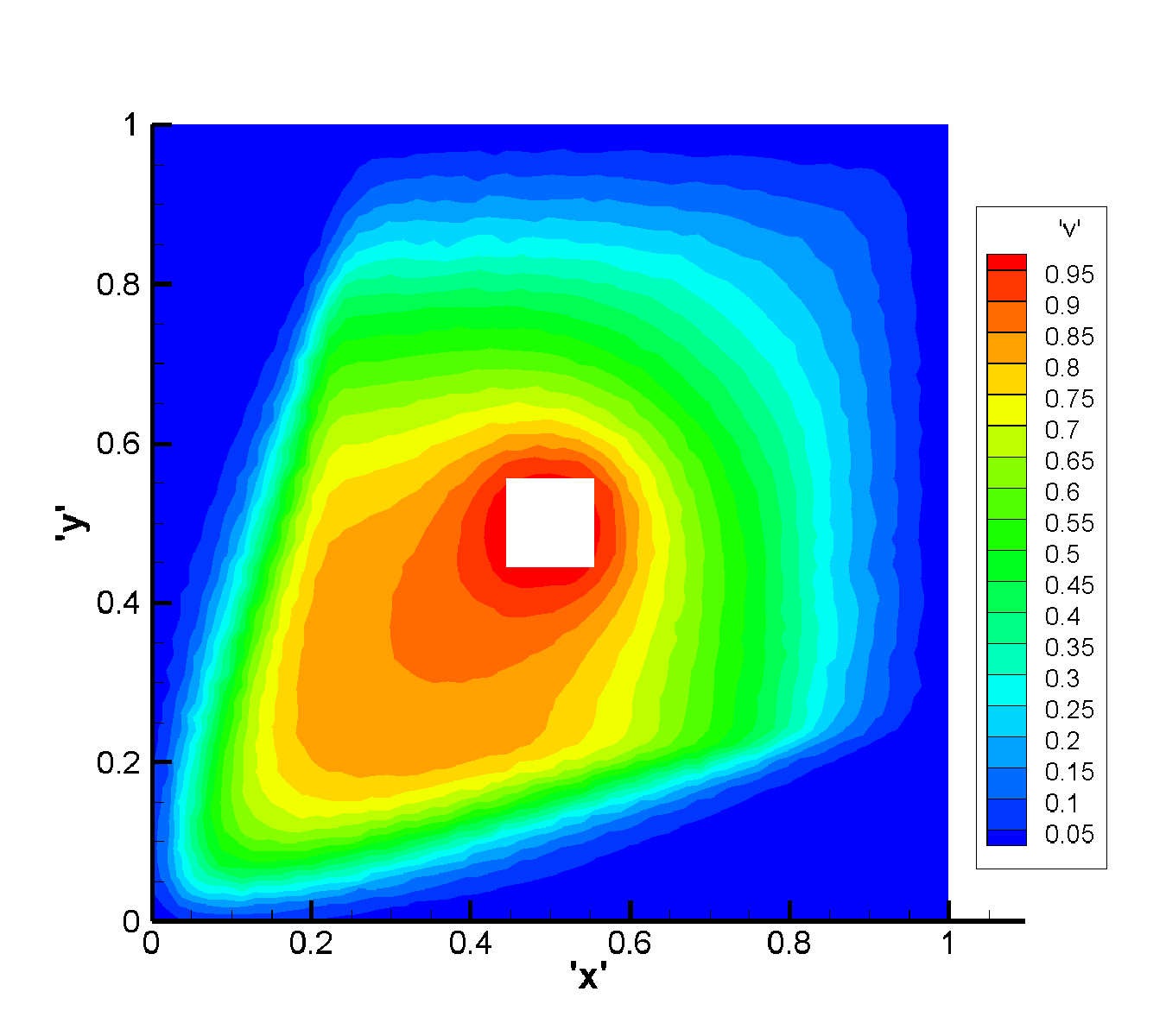}
		\caption{The numerical solution $V$  of the  IRP-preserving scheme   for Example 3 on the  random quadrilateral meshes ($V_{\min} = 1$, $V_{\max} =    1$).}  \label{exp3-2}
	\end{minipage}
	\\
	\begin{minipage}[t]{0.45\linewidth}
		\centering
		\includegraphics[width=2.2in]{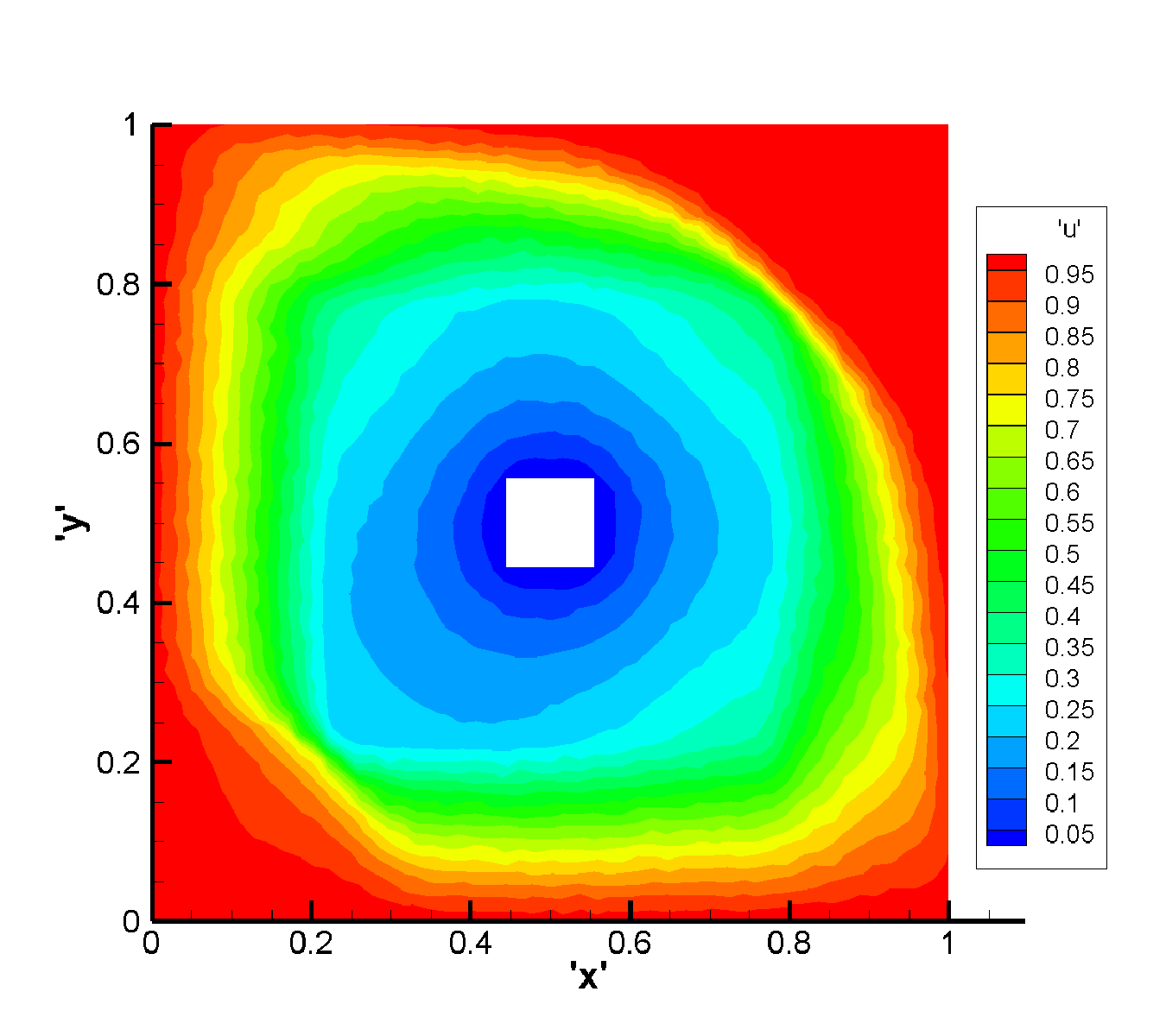}
		\caption{The numerical solution $U$ of the IRP-preserving scheme  for Example 3 on the random triangular meshes ($U_{\min} =  0$, $U_{\max} =  1$).}  \label{exp3-3}
	\end{minipage}
	\hspace{1em}
	\begin{minipage}[t]{0.45\linewidth}
		\centering
		\includegraphics[width=2.2in]{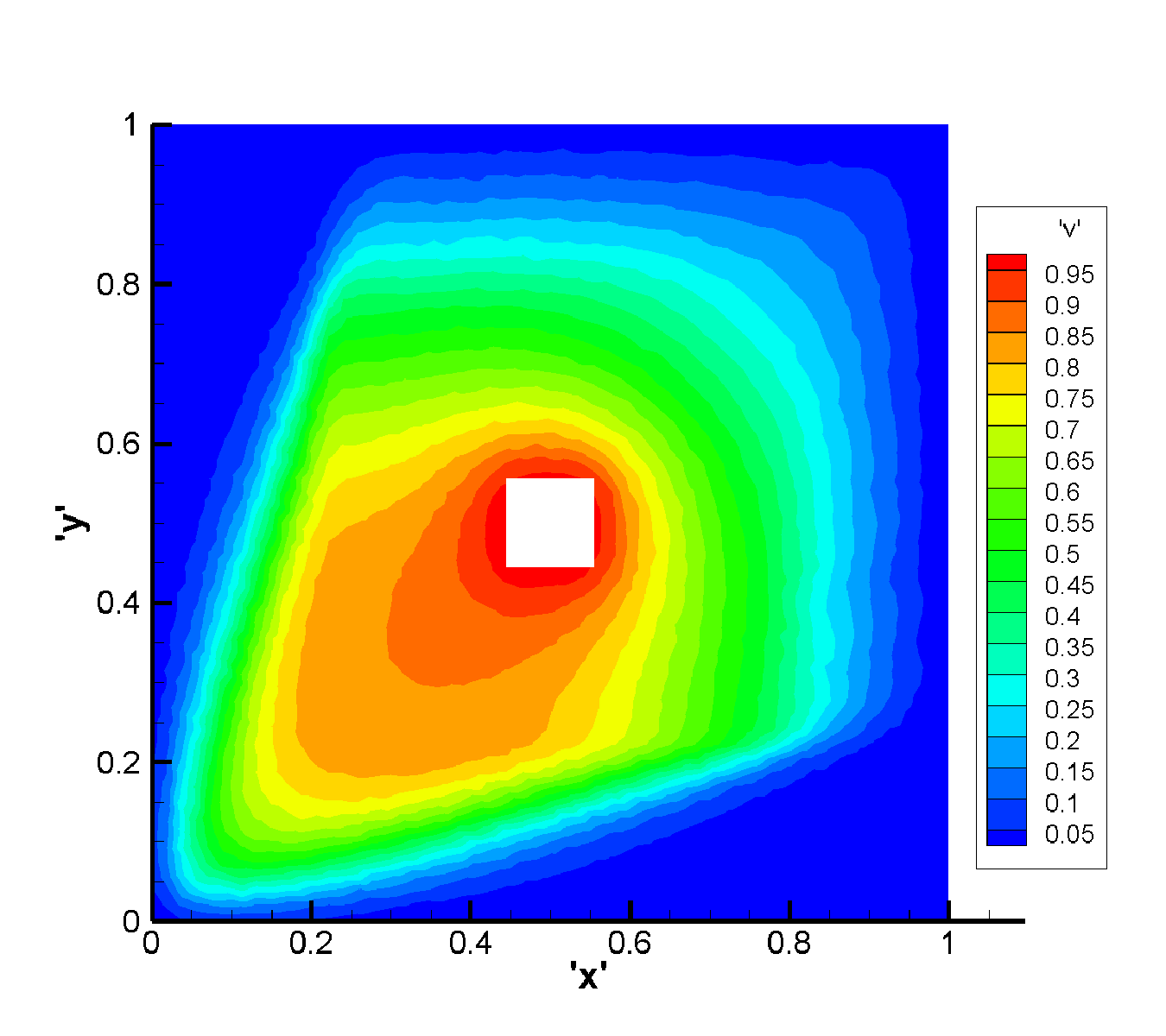}
		\caption{The numerical solution $V$  of the IRP-preserving scheme   for Example 3 on the  random triangular meshes ($V_{\min} =  0$, $V_{\max} =   1$).}  \label{exp3-4}
	\end{minipage}
\end{figure}

\subsection{Example 4}
In Example 4, we consider a simplified model of the  Belousov-Zhabotinski reaction on $\Omega =(0,1)\times (0,1)$, which is a classical model in non-equilibrium thermodynamics.  This model takes  the reaction  functions
\begin{align*}
	& f_1(u,v)  = u(a-bu-cv),
	\\
	&f_2(u,v)  =  -uv,
\end{align*}
where $a, b, c$ are positive constants. When $u\geq 0$ and $v\geq 0$, the $(f_ 1, f_2 )$ are quasimonotone nonincreasing, and $f_1$ and $f_2$ are Lipschitz continuous in any finite region of $(u,v)$. We take $a=1$, $b =2$,  $c=10$ and take the functions 
\begin{align*}
	\begin{aligned}
		u_0(x,y)=g_1(x,y,t)=\begin{cases} 
			0, & \text { if } 2x-y \leq-\frac{1}{16}, 
			\\
			8(2x-y)+0.5, & \text { if }-\frac{1}{16}<2x-y \leq \frac{1}{16}, 
			\\
			1, & \text { if } 2x-y>\frac{1}{16},
		\end{cases}  
	\end{aligned}
\end{align*}
and
\begin{align*}
	\begin{aligned}
		v_0(x,y)=g_2(x,y,t)=\begin{cases}  
			0, & \text { if } x+2y \leq \frac{15}{16}, \\
			16(x+2y)-15, & \text { if } \frac{15}{16}<x+2y \leq \frac{17}{16}, \\
			2, & \text { if } x+2y>\frac{17}{16}.
		\end{cases} 
	\end{aligned}
\end{align*}
A simple  calculation gives    $(u_0,v_0), (g_1,g_2)\in [0,1]\times[0,2]$,  the Lipschitz constant $\lambda = 35$.  According to  Lemma \ref{F-N-invariant_region}, the invariant region of  the exact solution  is  $[0,1]\times[0,2]$. 

We take $T=1$ and $\Delta t =1$E-3. We solve the problem using our scheme    and the nine-point (N-P) scheme   on the random quadrilateral meshes with $N_c=3600$ and random triangular meshes with $N_c=7200$, respectively.  The numerical solutions are presented  in Figures \ref{exp4-1}-\ref{exp4-2}. 
We denote   the maximum and minimum  of the solution vector $U$ as $U_{\max}$ and $U_{\min}$, respectively.  $V_{\max}$ and $V_{\min}$ are denoted similarly.
The numbers  of overshoots and undershoots of numerical solution are  denoted as $N_c^{o}$ and $N_c^{u}$, and the corresponding percentages are indicated by “pct”. 
Table \ref{exp4-quad} and Table \ref{exp4-tri}  summarize the maxima  and minima, the numbers of  overshoots and undershoots and their percentages of the two schemes, respectively. 
As shown in Tables \ref{exp4-quad}-\ref{exp4-tri}, the solution of our scheme on both meshes remains within the range of  $[0,1]\times[0,2]$, whereas the nine-point scheme fails to preserve the IRP.  In the latter case, we set the numerical solutions that fall outside the invariant region to gray for visual clarity.

\begin{table}[!htbp]  \centering  \caption{The maxima and minima  of the IRP-preserving scheme and nine-point scheme for Example 4 on random quadrilateral meshes.} 
	\begin{tabular}{lrrrrrr}   
		\hline
		method & \multicolumn{1}{l}{$U_{\max}$} &
		\multicolumn{1}{r}{$N_c^{o}$} & \multicolumn{1}{r}{pct} & \multicolumn{1}{r}{$U_{\min}$} & \multicolumn{1}{r}{$N_c^{u}$} & \multicolumn{1}{r}{pct} 
		\\
		\hline
		IRP   & 0     & 0     & 0.00\% & 1     & 0     & 0.00\% 
		\\
		NP    & 1.0290 & 562   & 15.61\% & -2.1341E-02 & 419   & 11.64\% 
		\\
		\hline
		& \multicolumn{1}{r}{$V_{\max}$} & \multicolumn{1}{r}{$N_c^{o}$} & \multicolumn{1}{r}{pct} & \multicolumn{1}{r}{$V_{\min}$} & \multicolumn{1}{r}{$N_c^{u}$} & \multicolumn{1}{r}{pct}  	  	  	
		\\
		\hline
		IRP   & 0     & 0     & 0.00\% & 2     & 0     & 0.00\% 
		\\
		NP    & 2.0641 & 655   & 18.19\% & -4.2454E-02 & 418   & 11.61\% 
		\\
		\hline
	\end{tabular}   \label{exp4-quad}
\end{table}%

\begin{table}[!htbp]  \centering  \caption{The  maxima and minima  of the IRP-preserving scheme and nine-point scheme for Example 4 on random  triangular meshes.} 
	\begin{tabular}{lrrrrrr}   
		\hline
		method & \multicolumn{1}{l}{$U_{\max}$} &
		\multicolumn{1}{r}{$N_c^{o}$} & \multicolumn{1}{r}{pct} & \multicolumn{1}{r}{$U_{\min}$} & \multicolumn{1}{r}{$N_c^{u}$} & \multicolumn{1}{r}{pct} 
		\\
		\hline
		IRP   & 0     & 0     & 0.00\% & 1     & 0     & 0.00\% 
		\\
		NP    & 1.0186& 699   & 9.71\% & -1.7108E-02 & 777   & 10.79\% 
		\\
		\hline
		& \multicolumn{1}{r}{$V_{\max}$} & \multicolumn{1}{r}{$N_c^{o}$} & \multicolumn{1}{r}{pct} & \multicolumn{1}{r}{$V_{\min}$} & \multicolumn{1}{r}{$N_c^{u}$} & \multicolumn{1}{r}{pct}  	  	  	
		\\
		\hline
		IRP   & 0     & 0     & 0.00\% & 2     & 0     & 0.00\% 
		\\
		NP    & 2.0694& 1535  & 21.32\% & -4.1207E-02 & 821   & 11.40\% 
		\\
		\hline
	\end{tabular}   \label{exp4-tri}
\end{table}%

\begin{figure}[!h]
	\begin{minipage}[t]{0.45\linewidth}
		\centering 
		\includegraphics[width=2.2in]{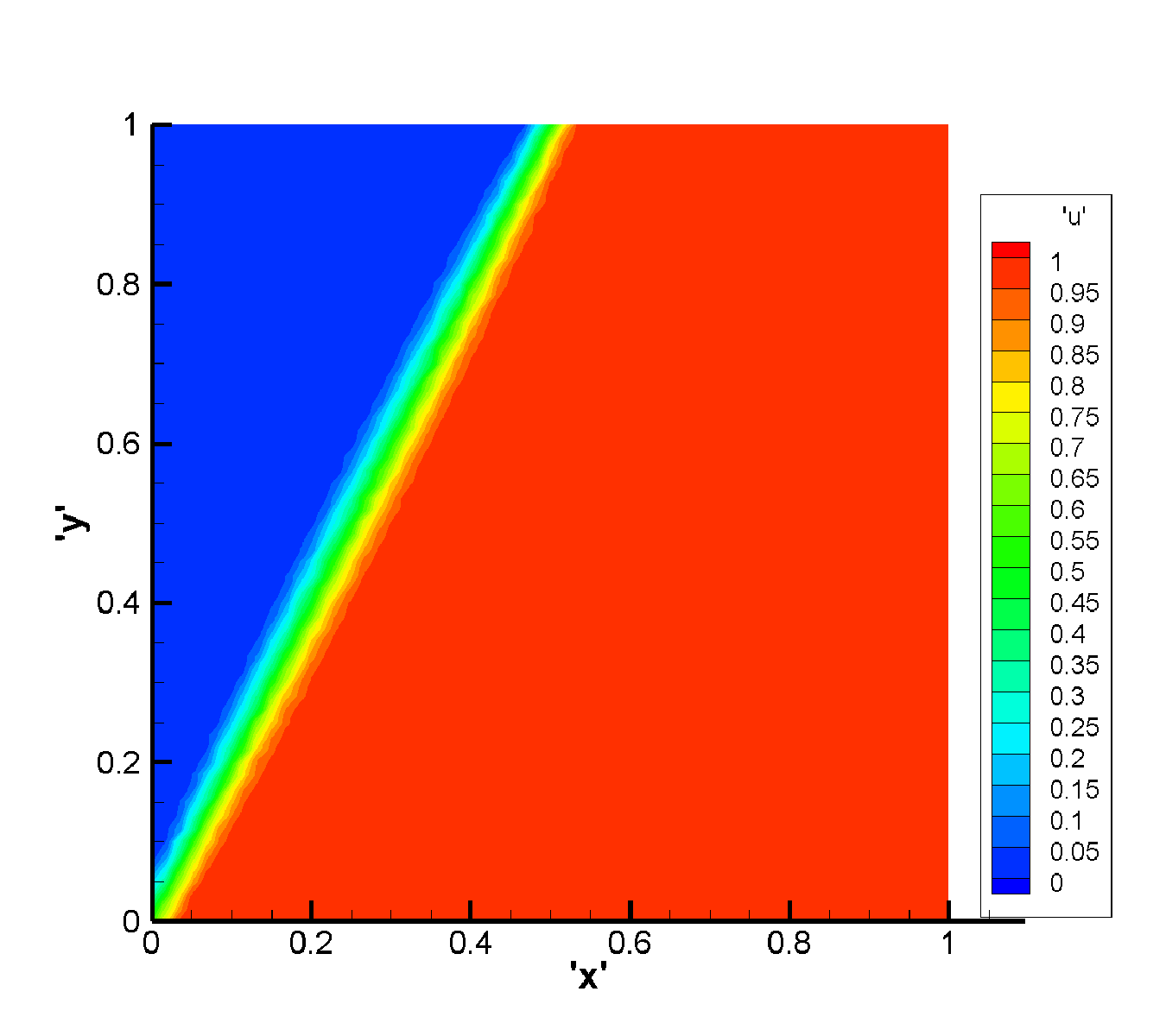}
		\caption{The numerical solution $U$ of the IRP-preserving scheme  for Example 4 on the random quadrilateral meshes ($U_{\min} = 0$, $U_{\max} =  1$).} \label{exp4-1}
	\end{minipage}
	\hspace{1em}
	\begin{minipage}[t]{0.45\linewidth}
		\centering
		\includegraphics[width=2.2in]{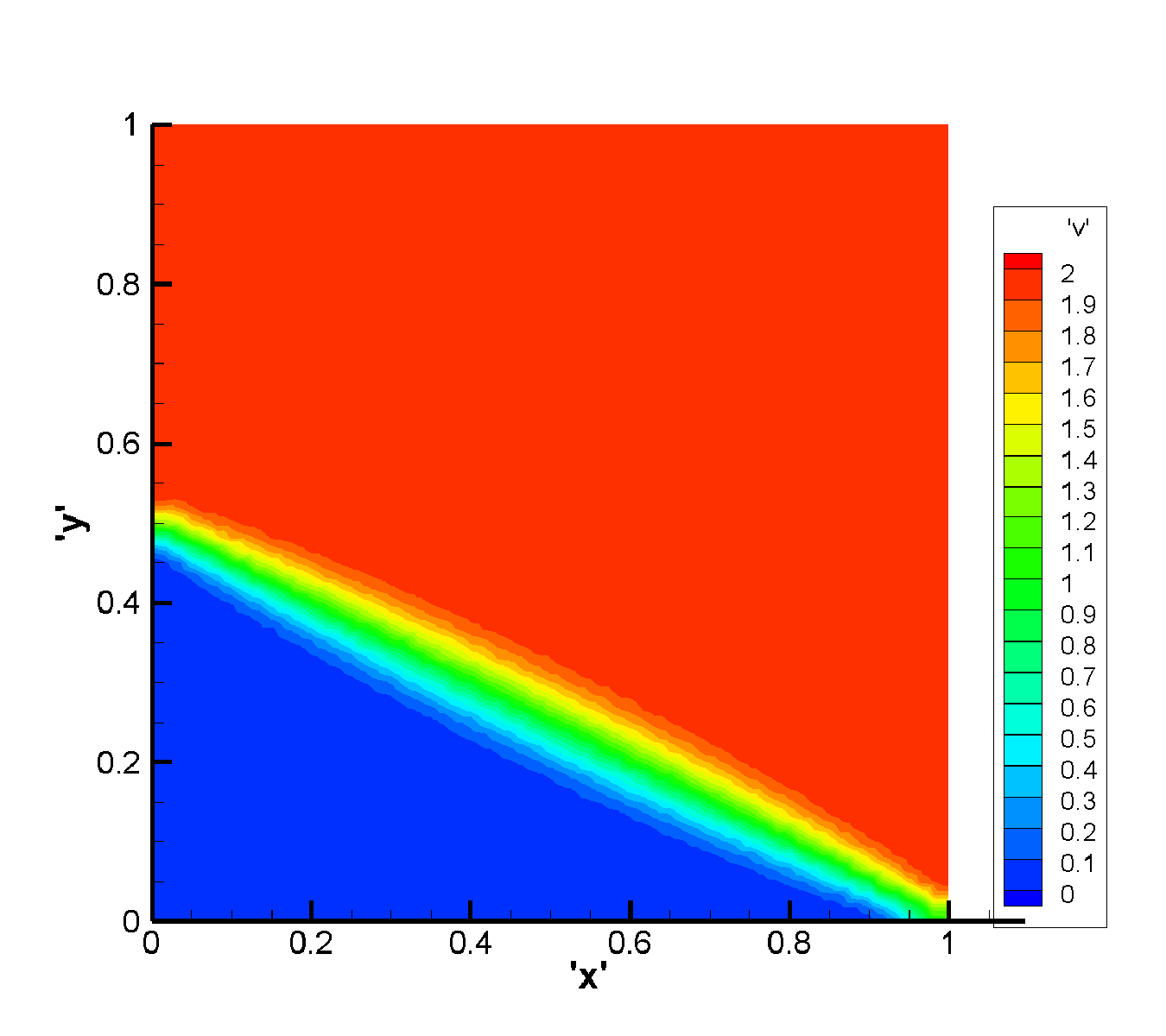}
		\caption{The numerical solution $V$  of the  IRP-preserving scheme   for Example 4 on the  random quadrilateral meshes ($V_{\min} = 0$, $V_{\max} =   2$).} 
	\end{minipage}
	\\
	\begin{minipage}[t]{0.45\linewidth}
		\centering
		\includegraphics[width=2.2in]{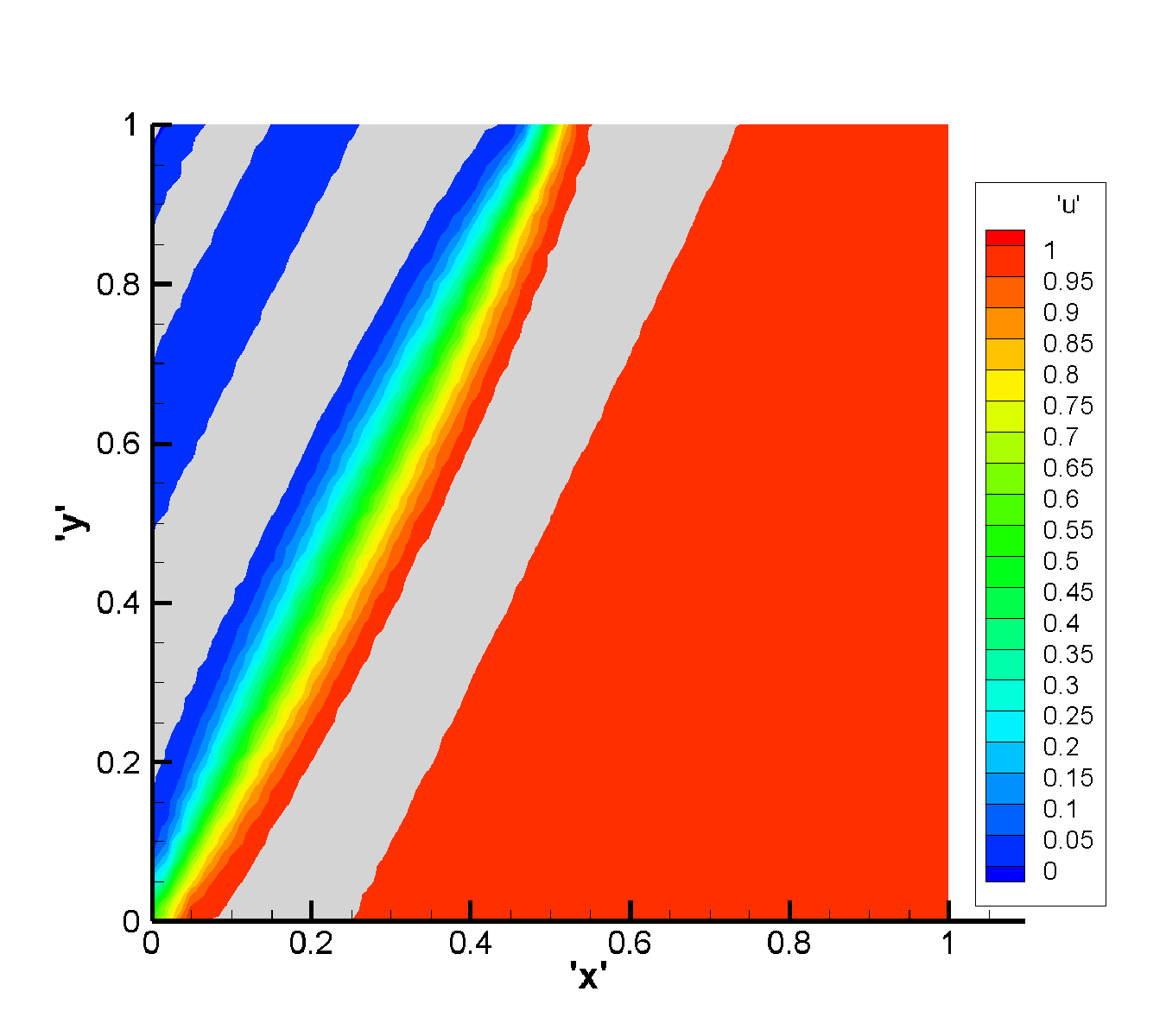}
		\caption{The numerical solution $U$ of the nine-point scheme  for Example 4 on the random quadrilateral meshes ($U_{\min} =  -2.1341$E-2, $U_{\max} = 1.0290$).} 
	\end{minipage}
	\hspace{1em}
	\begin{minipage}[t]{0.45\linewidth}
		\centering
		\includegraphics[width=2.2in]{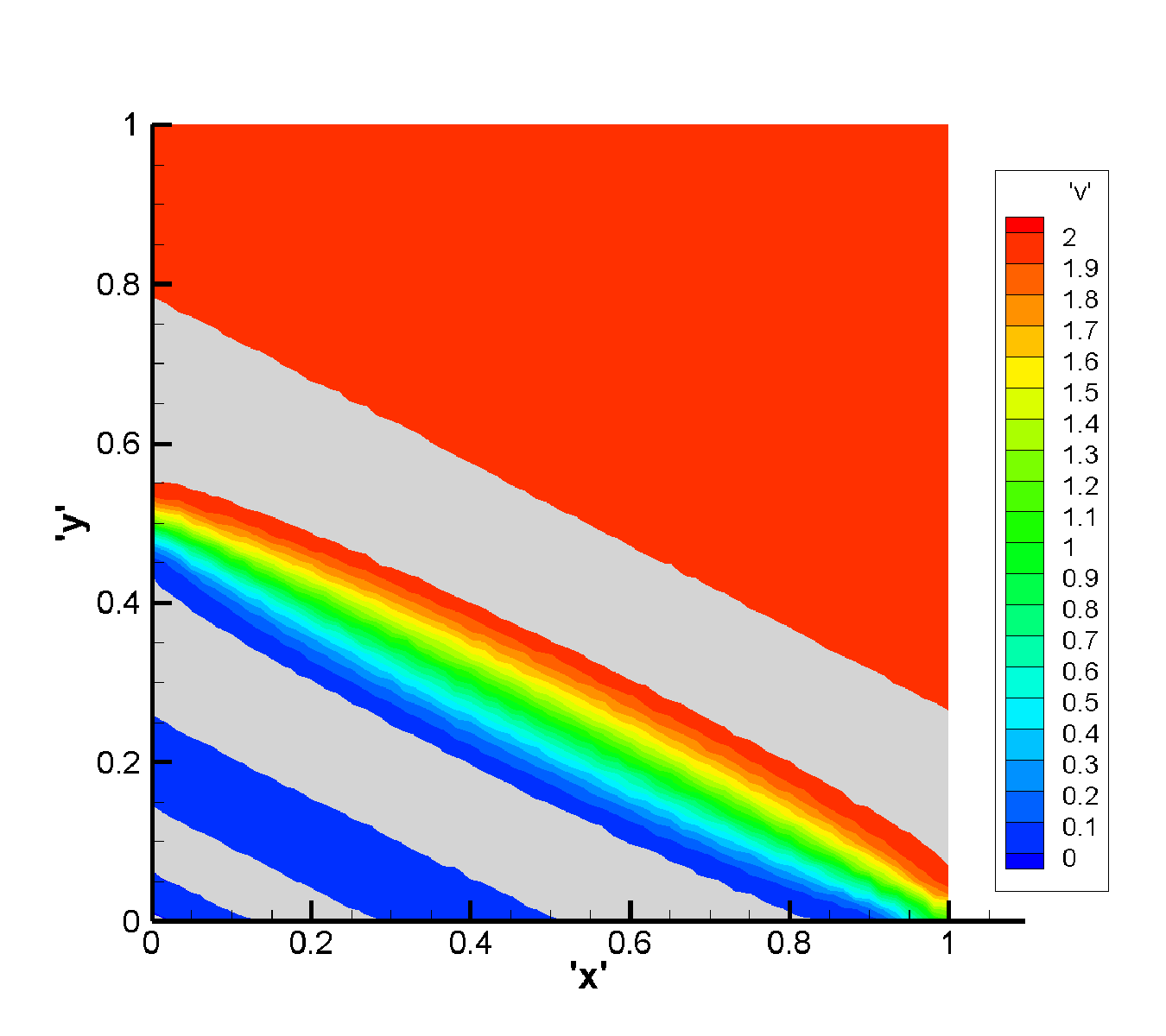}
		\caption{The numerical solution $V$  of the nine-point scheme   for Example 4 on the  random quadrilateral meshes ($V_{\min} =  -4.2454$E-2, $V_{\max} =     2.0641$).} 
	\end{minipage}
\end{figure}

\begin{figure}[!h]
	\begin{minipage}[t]{0.45\linewidth}
		\centering
		\includegraphics[width=2.2in]{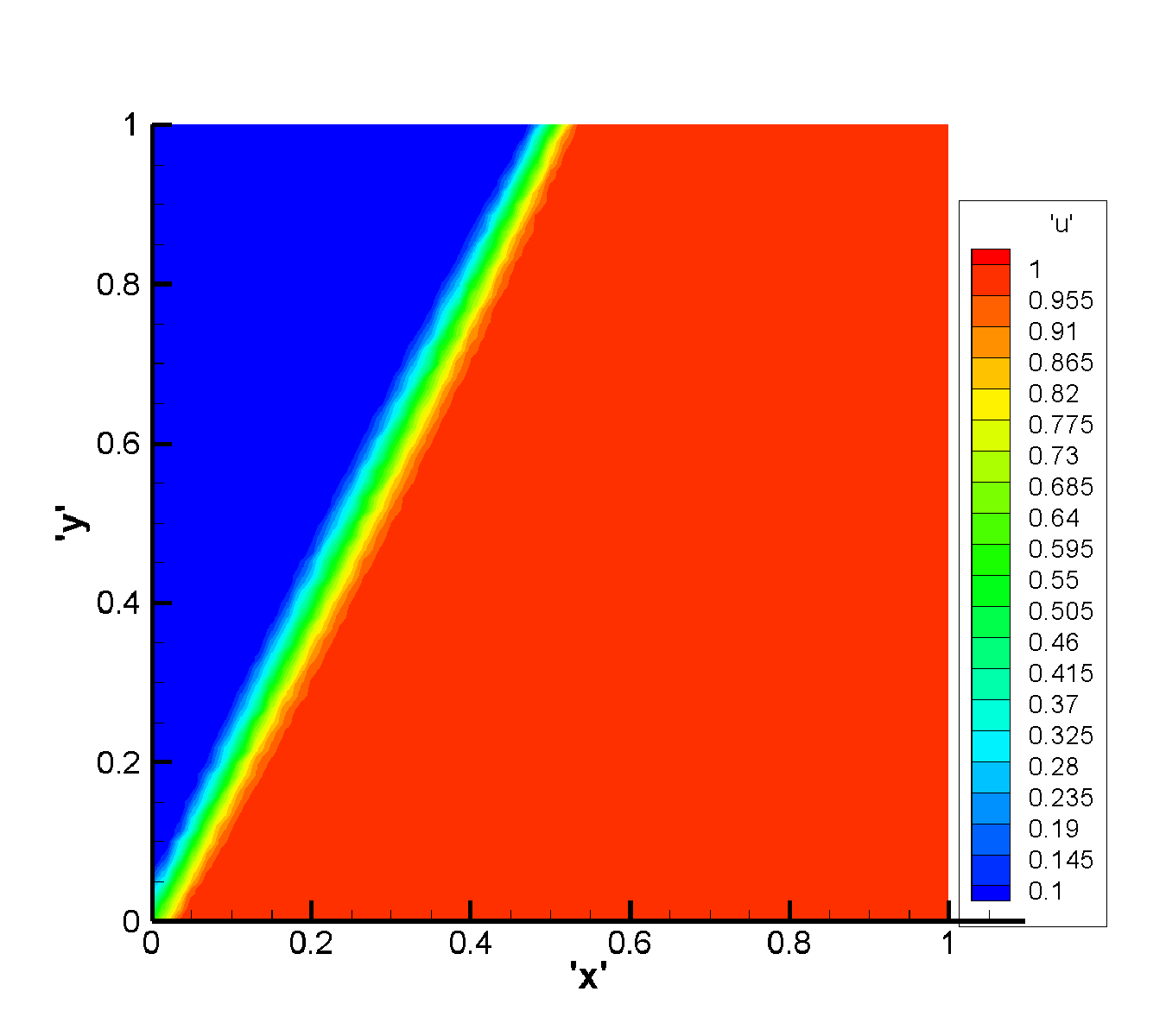}
		\caption{The numerical solution $U$ of the IRP-preserving scheme  for Example 4 on the random triangular meshes ($U_{\min} = 0$, $U_{\max} =  1$).} 
	\end{minipage}
	\hspace{1em}
	\begin{minipage}[t]{0.45\linewidth}
		\centering
		\includegraphics[width=2.2in]{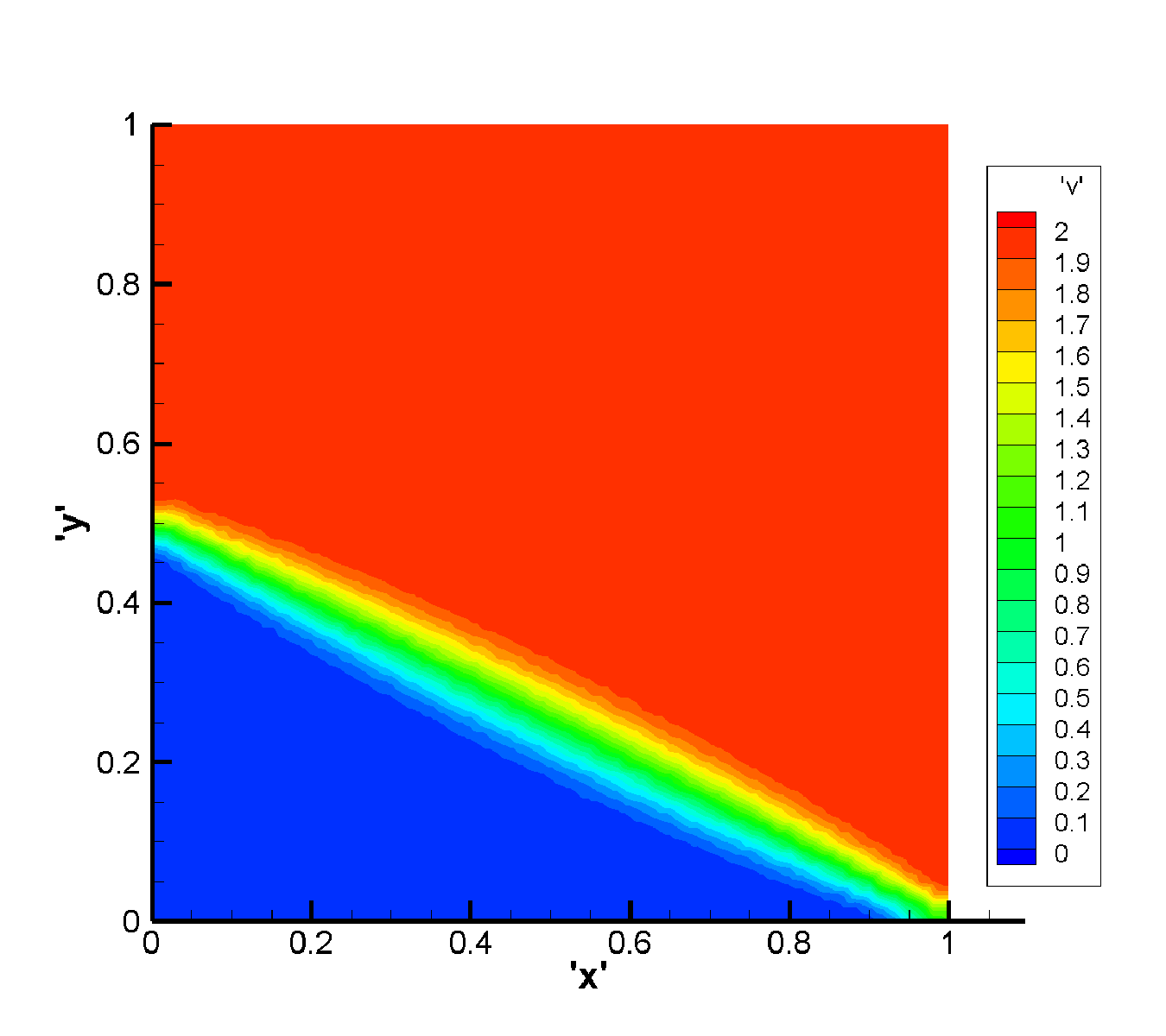}
		\caption{The numerical solution $V$  of the  IRP-preserving scheme   for Example 4 on the  random triangular meshes ($V_{\min} = 0$, $V_{\max} =   2$).} 
	\end{minipage}
	\\
	\begin{minipage}[t]{0.45\linewidth}
		\centering
		\includegraphics[width=2.2in]{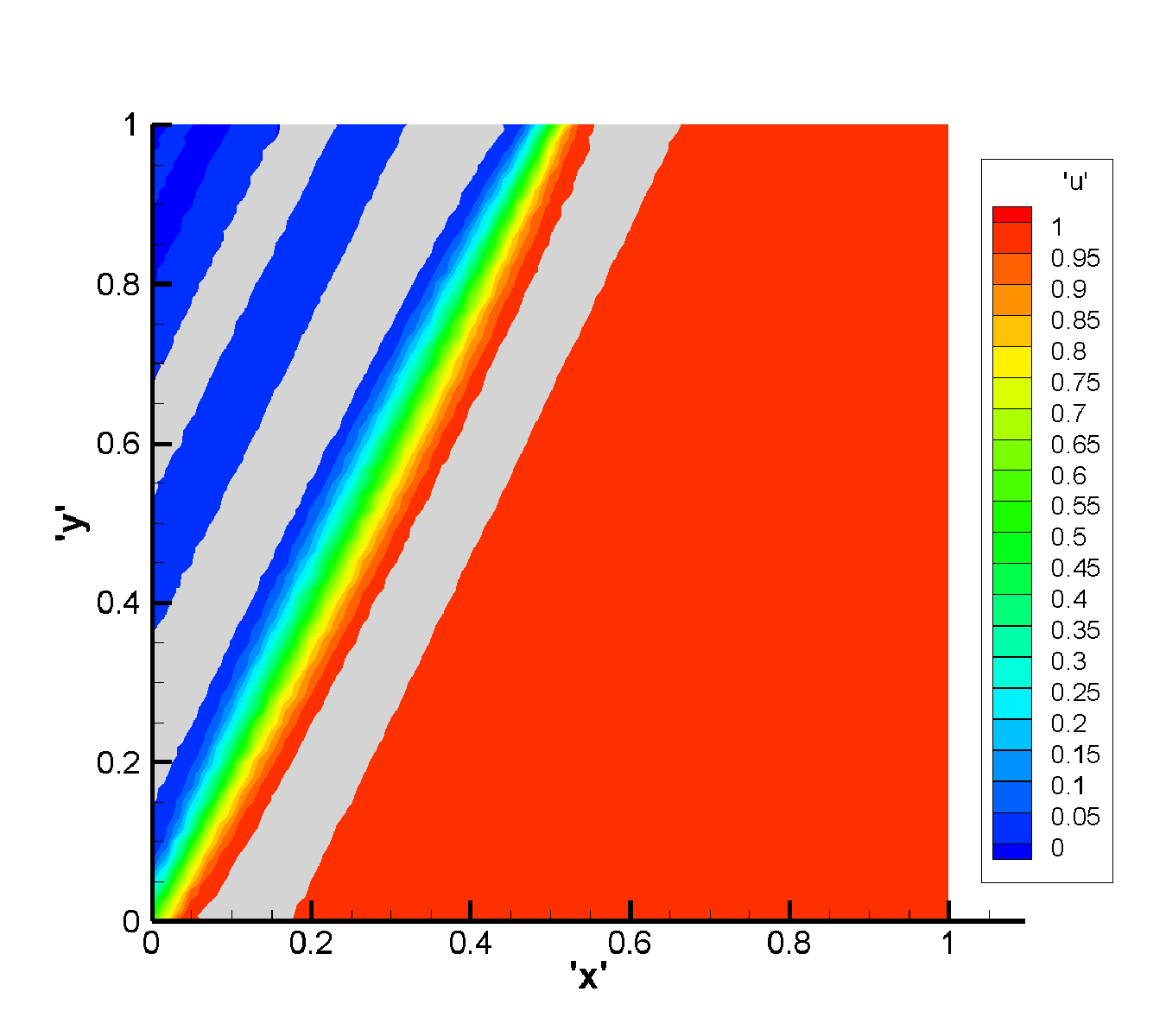}
		\caption{The numerical solution $U$ of the nine-point scheme  for Example 4 on the random triangular meshes ($U_{\min} =  1.7108$E-2, $U_{\max} = 1.0186$).} 
	\end{minipage}
	\hspace{1em}
	\begin{minipage}[t]{0.45\linewidth}
		\centering
		\includegraphics[width=2.2in]{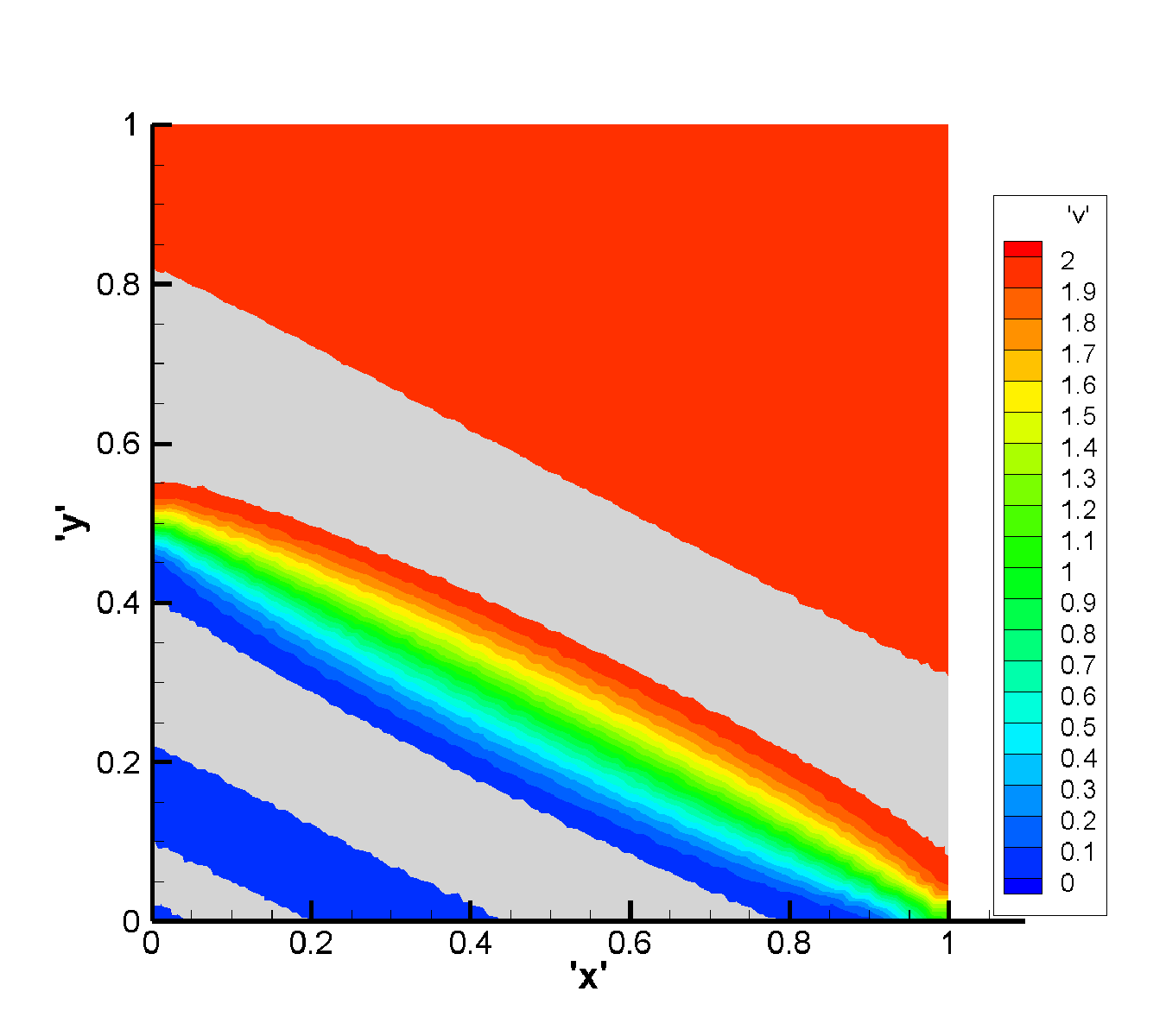}
		\caption{The numerical solution $V$  of the nine-point scheme   for Example 4 on the  random triangular meshes ($V_{\min} =  -4.1207$E-2, $V_{\max} =    2.0694$).} \label{exp4-2}
	\end{minipage}
\end{figure}

\section{Conclusion}
A finite volume method preserving the invariant region property is constructed  for   coupled   quasimonotone  reaction-diffusion  systems on general polygonal meshes. 
The backward Euler method and 
DMP-preserving finite volume scheme in \cite{Sheng2024} are employed to approximate the time derivatives  and the diffusion terms, respectively, which yields a nonlinear and conservative  scheme.  The  iterative scheme is constructed  to solve the nonlinear system. We prove that both the nonlinear scheme and the iterative method preserve the IRP for three kinds of quasimonotone systems.  Finally, some numerical examples are given to  illustrate the accuracy and the IRP-preserving property.

\section*{Acknowledgments}
This work is partially supported by
the National Natural Science Foundation of China (12201246), LCP Fund for Young Scholar
(6142A05QN22010), National Key R\&D Program of China (2020YFA0713601),
and by the Key Laboratory of Symbolic Computation and Knowledge Engineering of Ministry of Education, Jilin University, Changchun, 130012, P.R. China.


\end{document}